\documentclass[10pt]{article} 

\usepackage{url}
\usepackage{mathtools}
\usepackage{amssymb}
\usepackage{amsthm}
\usepackage{empheq}
\usepackage{scalerel}
\usepackage{titlesec}
\usepackage{latexsym}
\usepackage{enumitem}
\usepackage{eurosym}
\usepackage{dsfont}
\usepackage{appendix}
\usepackage{color} 
\usepackage[unicode]{hyperref}
\usepackage{frcursive}
\usepackage[utf8]{inputenc}
\usepackage[T1]{fontenc}
\usepackage{geometry}
\usepackage{multirow}
\usepackage{todonotes}
\usepackage{lmodern}
\usepackage{anyfontsize}
\usepackage{stmaryrd}
\usepackage{cleveref}
\usepackage{natbib}
\usepackage{makeidx}
\usepackage{etoolbox}
\usepackage[english]{babel}
\usepackage[english=british]{csquotes}

\setcitestyle{numbers,open={[},close={]}}

\definecolor{red}{rgb}{0.7,0.15,0.15}
\definecolor{green}{rgb}{0,0.5,0}
\definecolor{blue}{rgb}{0,0,0.7}
\hypersetup{colorlinks, linkcolor={red},citecolor={green}, urlcolor={blue}}

\makeatletter \@addtoreset{equation}{section}

\newcommand{\smallfont}[1]{\text{\fontsize{4}{4}\selectfont$#1$}}

\def\beq{\begin{eqnarray}}
\def\eeq{\end{eqnarray}}
\def\be*{\begin{eqnarray*}}
\def\ee*{\end{eqnarray*}}

\newtheorem{Theorem}{Theorem}[section]
\newtheorem{Lemma}[Theorem]{Lemma}
\newtheorem{Proposition}[Theorem]{Proposition}
\newtheorem{Corollary}[Theorem]{Corollary}
\newtheorem{Definition}[Theorem]{Definition}
\newtheorem{Remark}[Theorem]{Remark}

\def \E{\mathbb{E}}
\def \F{\mathbb{F}}

\def \L{\mathbb{L}}

\def \N{\mathbb{N}}

\def \P{\mathbb{P}}

\def \R{\mathbb{R}}

\def\Ac{{\cal A}}

\def\Fc{{\cal F}}

\def\Mc{{\cal M}}

\def \eps {\varepsilon}

\def \no {\noindent}

\def \1{\mathds{1}}

\begin{document}

\title{Randomness and early termination: what makes a game exciting?}


\author{Gaoyue {\sc Guo}\footnote{Université Paris-Saclay CentraleSupélec, Laboratoire MICS and CNRS FR-3487, France,  gaoyue.guo@centralesupelec.fr.} \and 
Sam D.\ {\sc Howison}\footnote{University of Oxford, Mathematical Institute, United Kingdom, howison@maths.ox.ac.uk.}
\and
Dylan {\sc Possama{\"i}}\footnote{ETH Z\"urich, Department of Mathematics, Switzerland, dylan.possamai@math.ethz.ch.} \and Christoph {\sc Reisinger}\footnote{University of Oxford, Mathematical Institute, United Kingdom, christoph.reisinger@maths.ox.ac.uk.}}



\maketitle

\begin{abstract}
In this paper we revisit an open problem posed by \citeauthor*{aldous_website} on the max-entropy win-probability martingale:
given two players of equal strength, such that the win-probability is a martingale diffusion, which of these processes has maximum entropy and hence gives the most excitement for the spectators? We study a terminal-boundary value problem for the nonlinear parabolic PDE $2\partial_t e(t,x)= \log(-\partial_{xx}e(t,x))$ derived by \citeauthor*{aldous_website} and prove its wellposedness and regularity of its solution by combining PDE analysis and probabilistic tools, in particular the
reformulation as a stochastic control problem with restricted control set, which allows us to deduce
strict ellipticity. 
We establish key qualitative properties of the solution including concavity, monotonicity, convergence to a steady state for long remaining time and the asymptotic behaviour shortly before the terminal time.
Moreover, we construct convergent numerical approximations.
The analytical and numerical results allow us to highlight the behaviour of the win-probability process
in the present case where the match may end early, in contrast to recent work by \citeauthor*{backhoff2023most} where the match always runs the full length.
\medskip

\no {\bf Keywords: }log diffusion PDE, most uncertain match, stochastic control, classical solution, weak solution, viscosity solution.
\end{abstract}

\section{Introduction}\label{sec:intro}

This paper is motivated by one of \citeauthor*{aldous_website}'s open problems about the characterisation of  the `most random martingale'. To be more precise, given a match of length $T>0$, for two players---or two teams---of equal level, denote by $X_t$ the winning probability of one player (or one team) at time $t\in[0,T]$. The win-probability $X\coloneqq (X_t)_{t\in [0,T]}$ is thus a martingale starting at $X_0=1/2$ and ending in $X_T \in \{0,1\}$. The question then becomes: which choice of $X$ leads to the most excitement to spectators, in the sense that the outcome remains uncertain until late? Taking relative entropy as a measure for `random', 
a heuristic derivation by \citeauthor*{aldous_notes} \cite{aldous_notes} leads to the following PDE for $e: [0,T]\times [0,1]\longrightarrow\R$, with $e(t,x)$ being the entropy given $X_t = x$
\begin{align}
\label{log-pde}
2\partial_t e(t,x)&= \log\big(-\partial_{xx}e(t,x)\big), \; (t,x) \in (0,T) \times (0,1)\coloneqq \Omega_T, \\
\label{tc}
e(T,x) &= 0,\; x \in (0,1), \\
\label{bc}
e(t,0) &= e(t,1) = 0, \;  t \in (0,T).
\end{align}
Specifically, at the end of \cite{aldous_website}, \citeauthor*{aldous_website} formulates the following open questions about \eqref{log-pde}--\eqref{tc}--\eqref{bc}:
\begin{enumerate}
\item[Q1:]
find an explicit solution of the PDE above, or at least prove existence and uniqueness of a solution; 
\item[Q2:]
find some of its qualitative properties;
\item[Q3:]
in particular, what is the distribution of $X_{T/2}$? 
\end{enumerate}
From our analysis, we can now answer these questions as follows:
\begin{enumerate}
\item[A1:]
in \Cref{thm:main} below, we ascertain the existence and uniqueness of a classical solution. Due to the boundary conditions, a separation \emph{ansatz} as in  \cite{backhoff2023most}
fails and we were unable to find a closed form solution, however,
we are able to provide an asymptotic formula for $t\ll T$  and $t\longrightarrow T$;
\item[A2:]
we were able to establish natural qualitative properties of the solution, such as monotonicity, concavity, and symmetry, summarised also in \Cref{thm:main}
and illustrated by a numerical solution in \Cref{fig:3d-con};
\item[A3:]
in absence of an analytical expression, we give a numerical approximation of the density at different times, including at $T/2$, in \Cref{fig:dens}, see also the discussion thereafter.
\end{enumerate}

A natural and unified approach to \citeauthor*{aldous_website}'s `most random martingale' \cite{aldous_website} comes from martingale optimal transport. Namely, fix a filtered probability space $(\Omega,\Fc,\F=(\Fc_t)_{t\in[0,T]},\P)$, and consider the optimisation problem
\begin{equation}\label{mot}
\sup_{X\in\Mc} \E^\P[F(X)],
\end{equation}
where $\Mc$ is some suitable set of $\R$-valued, $(\F,\P)$-martingales $X$ such that $\P[X_0=1/2]=1=\P[X_T\in \{0,1\}]$, and the choice of $F:\Mc\longrightarrow \R$ encapsulates some criterion that quantifies the game's attractiveness. In the recent paper \cite{backhoff2023most}, \citeauthor*{backhoff2023most} take $\Mc$ to be the collection of what they call `win martingales', that is to say martingales on $[0,T]$, which terminate in $\{0,1\}$, and whose quadratic variation is absolutely continuous with respect to Lebesgue measure. Up to enlarging the underlying probability space, this means that for any $X\in\Mc$, there exists an $\R$-valued volatility process $\sigma$ with 
\begin{equation}\label{eq:firstvol}
X_t =\frac 1 2 + \int_0^t\sigma_s   \, \mathrm{d} W_s,\; t \in [0,T].
\end{equation}
The criterion $F=F_0$\footnote{\citeauthor*{backhoff2023most} consider \[
G_0(X)\coloneqq F_0(X)-\scaleobj{.8}{\int^{T}_{0}} \sigma_t^2/2 \, \mathrm{d} t,\] as the objective function. Nevertheless, Itô's formula applied to $X_t^2$ yields $\E^\P\big[\int_0^T \sigma_t^2  \, \mathrm{d} t\big]=\E^\P[X_T^2-X_0^2]=1/4$ and thus $\E^\P[G_0(X)]=\E^\P[F_0(X)]-1/8$ for all $X\in \Mc$. Therefore, we do not distinguish $F_0$ and $G_0$ without any loss of generality.} is then defined specifically through the \emph{specific relative entropy} as
\begin{equation}
\label{obj_F_BB}
F_0(X)\coloneqq \frac 1 2 \int_0^T \big(\log\big(\sigma_t^2\big)  +1\big)  \, \mathrm{d} t,\; X\in\Mc.
\end{equation}
It is shown in \cite{backhoff2023most} that the optimiser for the above problem actually corresponds to a \emph{diffusion martingale}, that is to say that the optimal process $\sigma$ can by written in feedback form as $\bar\sigma(\cdot,X_\cdot)$ through the map $[0,T]\times\R\ni (t,x)\longmapsto \bar\sigma(t,x)\coloneqq \sin(\pi x)/(\pi\sqrt{T-t})\in\R$. Moreover, the associated entropy function is given by
\begin{equation}
\label{eqn:bare}
\bar e(t,x)\coloneqq (T-t)\bigg(\log\bigg(\frac{\sin(\pi x)}{\pi\sqrt{T-t}}\bigg)+\frac12\bigg),\; (t,x)\in[0,T]\times(0,1),
\end{equation}
which can be found, \emph{e.g.}, from the relationship $\partial_{xx} \bar e(t,x)=-1/\bar\sigma(t,x)^2$, satisfies \eqref{log-pde}--\eqref{tc}, but not \eqref{bc}, as in the original problem stated in \cite{aldous_notes}. In fact, it explodes at $x\in\{0,1\}$, at least for $t\in[0,T)$. We show later, in \Cref{prop:elliptic} and \Cref{lem:comp}, that the solution $e$ to \eqref{log-pde}--\eqref{tc}--\eqref{bc} admits a probabilistic representation that has a similar form to \eqref{mot},
where the essential difference arises in the choice of $F$, \emph{i.e.}
\begin{equation}
\label{obj_F}
F(X)=\frac 1 2 \int_0^{\min(T,\tau)} \big(\log\big(\sigma(t,X_t)^2\big)  +1\big)  \, \mathrm{d} t,
\end{equation}
where $\tau\coloneqq \inf \{t\in [0,T]:  X_t\notin (0,1)\}$ is the first exit time of $X$ from $(0,1)$. Here, the match under our consideration may end strictly prior to $T$, which for instance would be the case in a boxing match.
As a consequence, the entropy function satisfies the absorbing boundary conditions \eqref{bc}, whereas $\bar\sigma(t,x)$ from \cite{backhoff2023most} vanishes at the boundaries and prevents the process from
hitting the boundary prior to the end time. Viewed differently, \eqref{obj_F_BB} imposes a penalty of $-\infty$ if the process gets absorbed at the boundary
prior to time $T$.

\medskip

\Cref{fig:dens} highlights, for $T=1$, the difference between optimal matches which may or may not terminate early. Shown left is the probability density $q(t,\cdot)$ of $X_t$ for $\sigma^\star$ found by \cite{backhoff2023most}, for three different times $t$; shown right is the sub-probability density
of $X_t$ when it does not hit the boundary, under the optimal volatility in \citeauthor*{aldous_website}'s original model, found by computations explained in \Cref{sec:numerics}. The density is in both cases approximated by a finite-difference solution of the corresponding Kolmogorov forward equation.
\begin{figure}[ht!]
\hspace{-0.2 cm}
\includegraphics[width=0.50\columnwidth]{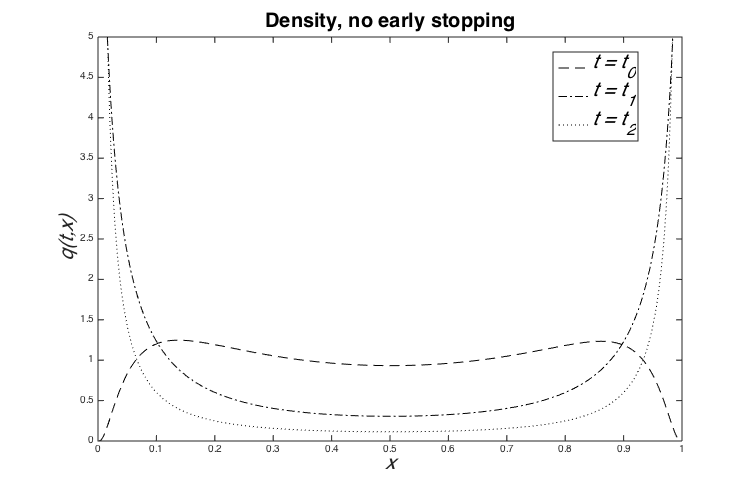}
\hspace{-0.5 cm}
\includegraphics[width=0.50\columnwidth]{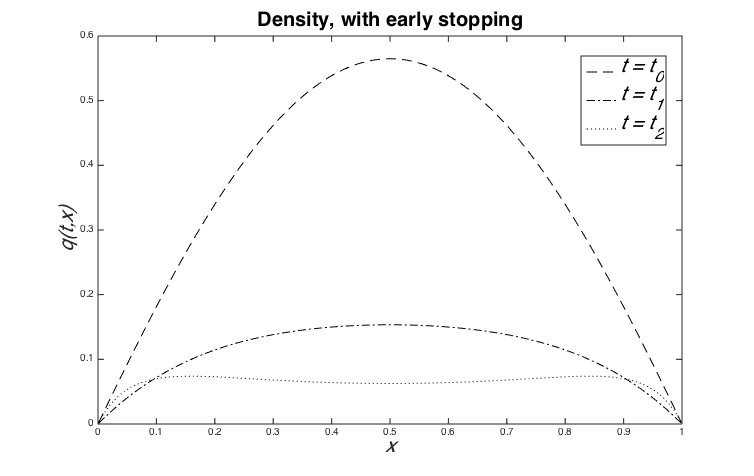}
\caption{\small The (sub-)probability density $q(t,\cdot)$ of $X_t$,  for $t_0=0.5$, $t_1=0.9$, $t_2=0.99$, under the optimal volatility from \cite{backhoff2023most} (\emph{i.e.},  the match never finishes early, left) and ours (match may finish early, right).}
\label{fig:dens}
\end{figure}
First, in \citeauthor*{aldous_website}'s model, there is a significant probability that the match started at $0$ terminates before $1$, \emph{i.e.}, $X$ hits a boundary before time $1$, namely with $63$\% chance the match is over before $t_0=0.5$, with $88$\% before $t_1=0.9$, and with $93$\% before $t_2=0.99$ (these being found by computing $1$ minus the integral of $q(t,\cdot)$ over $(0,1)$).
Until close to the end, provided the match has not ended, the highest density is found in the centre. Under the model in \cite{backhoff2023most}, where the match always runs until time $1$, \emph{i.e.}, the process is forced to stay in $(0,1)$, probability mass is fairly evenly distributed at time $t_0=0.5$, and then accumulates close to the boundaries, terminating in two atomic masses of weight $0.5$ at $0$ and $1$ at time $1$.

\begin{Remark}
One could argue that the {\rm`}most random match{\rm'} is not necessarily equivalent to the {\rm`}most exciting match{\rm'}, where the former emphasises the outcome uncertainty and the latter could express, say, the excitement resulting from sudden changes within the game, \emph{e.g.},  the more oscillating the martingale $X$, the more exciting the match.\footnote{A typical exciting match is the 2008/2009 Champions League semi-final between Chelsea and Liverpool. First leg 3:1; second leg, aggregate: 3:2, 3:3, 4:3, 5:3, 6:3, 6:4, 6:5, 7:5.}
To address this issue, one might consider the martingale optimal transport problem \eqref{mot} with 
\[
F(X)=\max\big\{U_{\eps}(X), D_{\eps}(X)\big\} \; \mbox{\rm or} \; F(X)=\mathbf{1}_{\{ \max(U_{\smallfont\eps}(X), D_{\smallfont\eps}(X))\ge n\}},
\]
 where $U_{\eps}(X)$ $($resp. $D_{\eps}(X))$ denotes the number of up-crossings $($resp. down-crossings$)$ of $X$ on $[\eps, 1-\eps]$ for some $\eps \in (0,1/2)$, and $n\in\N^\star$ is some psychological threshold of the spectators,  
  see \emph{e.g.}, {\rm\citeauthor*{leike2014indefinitely} \cite{leike2014indefinitely}}. 
\end{Remark}

In what follows, we adopt \citeauthor*{aldous_website}'s criterion 
(see the heuristic derivation in \cite{aldous_notes}) and focus on \eqref{log-pde}--\eqref{tc}--\eqref{bc}. Differentiating formally \eqref{log-pde} twice with respect to $x$ and setting $p(t,x)\coloneqq -\partial_{xx}e(1-t,x)$, we obtain
\begin{align}
\label{log-pde2}
2\partial_tp(t,x) &=\partial_{xx} \big(\log(p(t,x)\big),\; (t,x) \in \Omega_T, \\
\label{tc2}
p(0,x) &= 0, \;  x \in (0,1), \\
\label{bc2}
p(t,0) &= p(t,1) = 1,\;   t \in (0,T),
\end{align}
where a smooth-fit, which we will justify rigorously later, yields $p(t,0)=-\partial_{xx}e(1-t,0)=\exp(\partial_te(1-t,0))=1$ (resp. $p(t,1)=-\partial_{xx}e(1-t,1)=\exp(\partial_te(1-t,1))=1$). \Cref{log-pde2}, known as the \emph{logarithmic diffusion equation}, has mathematical significance as it arises in the study of Ricci flow (see, \emph{e.g.}, \citeauthor*{hamilton1988ricci} \cite{hamilton1988ricci}, \citeauthor*{wu1993ricci} \cite{wu1993ricci}), and physical significance in connection with the dynamics of thin liquid films (see, \emph{e.g.}, \citeauthor*{bertozzi1994singularities} \cite{bertozzi1994singularities}, \citeauthor*{burelbach1988nonlinear} \cite{burelbach1988nonlinear}, \citeauthor*{vazquez1992nonexistence} \cite{vazquez1992nonexistence}) and as a model for the limiting density in the kinetics of two gases moving against each other and obeying the Boltzmann equation (see, \emph{e.g.}, \citeauthor*{GS2007} \cite{GS2007}, \citeauthor*{salvarani2009asymptotic} \cite{salvarani2009asymptotic}  and  \citeauthor*{davis1996some} \cite{davis1996some}).  

\medskip
To the best of our knowledge and in contrast to \eqref{log-pde2}, studies of the theory and approximation of \eqref{log-pde} are close to non-existent so far. 
The only related problem we are aware of is the parabolic Monge--Amp\`ere equation---for the study of the K\"ahler--Ricci flow---together with an initial condition (without boundary conditions; see \citeauthor*{imbert2013introduction} \cite[page 10 and Remark 2.3.2]{imbert2013introduction}). The present paper shows the well-posedness of \eqref{log-pde}--\eqref{tc}--\eqref{bc} and characterises its properties, filling in particular the gap between \eqref{log-pde} and \eqref{log-pde2} by a representation formula. We will also work heavily with the control formulation of the
problem, specifically to show a lower positive bound for the optimal control, which allows us to deduce strict ellipticity of the PDE. Before stating our main result, we introduce the following definition of \emph{weak solution} to the logarithmic diffusion equation in \citeauthor*{salvarani2009asymptotic} \cite{salvarani2009asymptotic}. 

\begin{Definition}
\label{def:weak}
Let $g\in H^1((0,1)) \cap \L^\infty((0,1))$ be non-negative and $c>0$. Then $u:\overline \Omega_T\longrightarrow\R$ is said to be a weak solution to 
\begin{align*}
2\partial_tu(t,x) =\partial_{xx} \big(\log(u(t,x)\big),\; (t,x) \in \Omega_T, \; u(0,x) = g(x), \;  x \in (0,1), \; u(t,0) = u(t,1) = c,\;   t \in (0,T),
\end{align*}
if the following conditions hold
\begin{enumerate}
\item[$(i)$] $u(t,0)= u(t,1) = c$ for all $t \in (0,T);$
\item[$(ii)$] $u\in \L^2(\Omega_T)\cap C([0,T], H^1((0,1))$ is non-negative$;$
\label{item:sobolev}
\item[$(iii)$]  $\log(u)\in \L^1_{\mathrm{loc}}(\Omega_T)$, $\partial_x\log(u)\in \L^2(\Omega_T);$
\item[$(iv)$] the identity
\[
\int_0^T\int_0^1 \bigg(u(t,x)\partial_t\phi(t,x) -\frac 1 2 \partial_x\big(\log(u(t,x))\big)\partial_x\phi(t,x) \bigg) \, \mathrm{d} x\, \mathrm{d} t + \int_0^1 g(x)\phi(0,x) \, \mathrm{d} x  =0,
\]
holds for all $\phi\in H^1(\Omega_T)\cap C(\overline \Omega_T)$ vanishing at $t=T$, $x=0$ and $x=1$.
\end{enumerate}
\end{Definition}

Define $e_\infty(x)\coloneqq x(1-x)/2,\; x\in[0,1],$ the stationary solution to \eqref{log-pde}--\eqref{bc}. Our main result is the following. 
\begin{Theorem}
\label{thm:main}
There exists a unique solution $e\in C^{1,2}(\Omega_T)$ to \eqref{log-pde}--\eqref{tc}--\eqref{bc}. Moreover, it holds that
\begin{enumerate}[label=$(\roman*)$]
  \item \label{it_bdd} $0\le e(t,x)\le e_\infty(x)$ for all $(t,x)\in \overline \Omega_T;$
    \item \label{it_int} $e$ admits the following integral representation
   \begin{equation}\label{candidate}
e(t,x)=-\int_0^x\int_0^y p(T-t,z)  \, \mathrm{d} z\, \mathrm{d} y + x\int_0^1\int_0^y p(T-t,z)   \, \mathrm{d} z\, \mathrm{d} y,\; (t,x) \in  \overline \Omega_T,
\end{equation}
where $p$ is the unique weak solution to \eqref{log-pde2}--\eqref{tc2}--\eqref{bc2}$;$
  \item \label{it_decr} $t\longmapsto e(t,x)$ is non-increasing for all $x\in [0,1];$
  \item \label{it_ccv} $x\longmapsto e(t,x)$ is concave and symmetric with respect to $x=1/2$, for all $t\in [0,T]$.
 \end{enumerate}
\end{Theorem}

These key features qualitatively describe the graph seen in \Cref{fig:3d-con} below, where we set again $T=1$. Notice also that $\partial_t e(t,x) \longrightarrow -\infty$ for $t \uparrow 1$, which is necessary for $\partial_{xx}e(x,t)$ being continuous in time at $t=1$.
\begin{figure}[ht!]
\hspace{-0.2 cm}
\includegraphics[width=0.50\columnwidth]{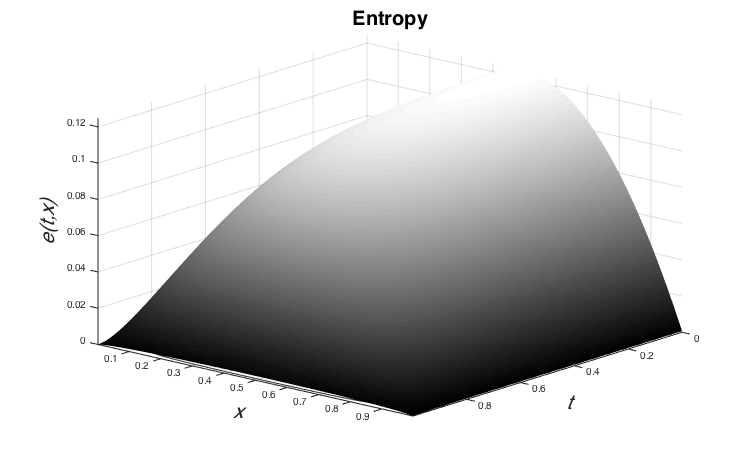}
\hspace{-0.5 cm}
\includegraphics[width=0.50\columnwidth]{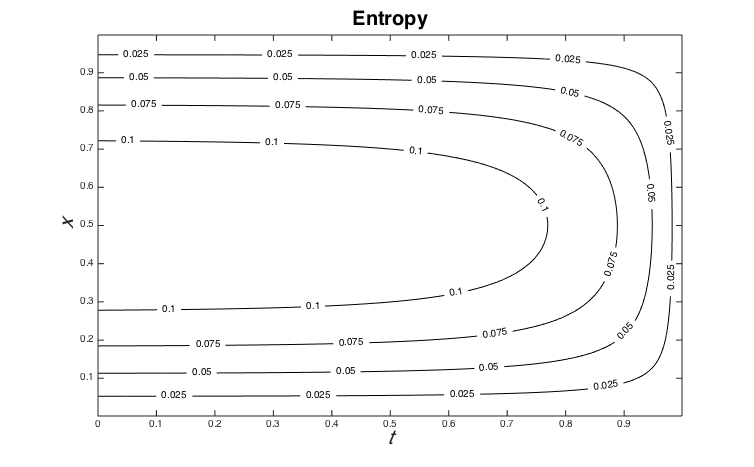}
\caption{\small Entropy $e$ as surface and contour plot.}
\label{fig:3d-con}
\end{figure}

\medskip
We observe that the curves $x\longmapsto e(t,x)$ (resp. $x\longmapsto \partial_{xx}e(t,x)$ in \Cref{fig:uxx-vol} at the end of the paper) 
evolve very slowly for $t\in [0,0.7]$ and vary increasingly fast for $t\in [0.9,1)$. This indicates that the criterion ‘most random martingale’ is consistent with the notion of ‘most uncertain match’, where the outcome (conditional on the match not having ended prematurely) is still uncertain until near the end; see also \Cref{fig:dens}. 

\medskip
In Section \ref{sec:asympt} we will derive an asymptotic approximation of $p$, $e$ and $\sigma^\star$
as $t \longrightarrow T$, which shows that in an `outer' region away from the boundaries our solution has the same asymptotic behaviour as the one from \cite{backhoff2023most}, while in an `inner' region close to the boundaries  the solution transitions from the boundary values to the `outer' solution through
a smooth function given in semi-analytic form.

\medskip
Our last main result pertains to \Cref{eq:firstvol} whenever we consider the corresponding optimal volatility. We thus denote the candidate optimal volatility $\sigma^\star: [0,T]\times \mathbb{R} \longmapsto [0,\infty]$ with $\sigma^\star(t,x) = p(T-t,x)^{-1/2}$
for $x \in [0,1]$ and extend it continuously by 1 for $x \notin [0,1]$, $t<T$, if needed.
The fact that  $p(0, \cdot) = 0$ implies $\sigma^\star(t,\cdot) \longrightarrow \infty$ as $t\longrightarrow T$; in fact, the asymptotic analysis
in Section \ref{sec:asympt} reveals that the singularity at $T$ has the same order as $\bar\sigma$.
We show that  $\sigma^\star$ is sufficiently well-behaved prior to $T$ so that the following SDE is indeed well-posed:
\begin{equation}
\label{sde_opt}
X_t =\frac 1 2 + \int_0^t\sigma^\star(s,X_s)   \, \mathrm{d} W_s.
\end{equation}

\begin{Theorem}\label{th:wellposedsde}
There is a unique strong solution $X^\star$ to \eqref{sde_opt}. Moreover, $\mathbb{P}[\tau^\star \le T] = 1$,
for $\tau^\star \coloneqq \inf \{t\in [0,T]:  X^\star_t\notin (0,1)\}$.
\end{Theorem}

\section{A control problem representation}\label{sec:control}

Inspired by the Legendre transform of the concave function $(-\infty,0)\ni x\longmapsto \log(-x)\in\R$, \emph{i.e.}
\[
\inf_{a\ge 0} \big\{-a x - \log a - 1\big\}=
     -\infty\mathbf{1}_{\{x\ge 0\}}+
       \log(-x)\mathbf{1}_{\{x<0\}},
\]
 the PDE \eqref{log-pde} can be rewritten as the following Hamilton--Jacobi--Bellman equation 
\begin{equation} \label{transform-pde}
 \partial_te(t,x) = \frac12\inf_{a\ge 0} \big\{-a \partial_{xx}e(t,x) - \log a - 1\big\}, \; (t,x) \in \Omega_T. 
\end{equation}
Now let us introduce the probabilistic counterpart of  \eqref{transform-pde}, and consider the following stochastic control problem. 
 Let $(\Omega,\mathcal F, \mathbb P)$ be a probability space on which a one dimensional $(\F,\P)$--Brownian motion $W$ is defined, where $\F$ is the $\P$-completion of the natural filtration of $W$. Denote by $\mathcal A$  the set of $\F$--progressively measurable processes $\alpha=(\alpha_t)_{t\ge 0}$ taking values in $\mathbb R_+$ such that $\mathbb E^\P\big[\int_0^t \alpha_s \, \mathrm{d} s\big]<\infty,\; \forall t\in[0,T].$ For $\alpha\in\mathcal A$, $t\in [0,T]$ and $x\in [0,1]$, denote by $X^{\alpha,t,x}=(X_s^{\alpha,t,x})_{s\in [t,T]}$ the controlled process given as 
\[
X_s^{\alpha,t,x}\coloneqq x+ \int_t ^s \sqrt{\alpha_r}\, \mathrm{d} W_r.
\]
Define further the reward function 
\[
J(t,x,\alpha)\coloneqq \E^\P\bigg[ \frac{1}{2} \int_t^{\min\{T,\tau^{\smallfont{\alpha}\smallfont{,}\smallfont{t}\smallfont{,}\smallfont{x}}\}}\big(1+\log(\alpha_s)\big)\, \mathrm{d} s\bigg],\; \mbox{with}\; \tau^{\alpha,t,x}\coloneqq \inf\big\{s\ge t:  X^{\alpha,t,x}_s\notin (0,1)\big\},
\]
and the value function 
\begin{equation} \label{def:val-fun} 
v(t,x)\coloneqq \sup_{\alpha\in\Ac} J(t,x,\alpha),\;  (t,x)\in \overline\Omega_T.
\end{equation} 
It follows by definition that  $v(T,x)=v(t,0)=v(t,1)=0$ for all $(t,x)\in \overline\Omega_T$. Provided that there exists a classical solution with $\partial_{xx}e(x,t) < 0$, the minimiser of the right-hand side in  \eqref{transform-pde} is then $a^\star(t,x) \coloneqq  -1/\partial_{xx}e(t,x)$, which makes the link with the choice of the optimal volatility function $\sigma^\star$ as explained in \Cref{sec:intro}. In particular, $a^\star(t,0) = a^\star(t,1) = 1$. 

\medskip
Our goal here is to identify $v$ as the unique bounded viscosity solution of \eqref{log-pde}--\eqref{bc}--\eqref{tc}, and to prove further that $v$ is smooth enough. In order to do so, the current section focuses on first proving that $v$ is indeed bounded, and second that $v(t,x)$ can never be achieved by $J(t,x,\alpha)$ when the control $\alpha$ is too close to $0$. The latter property is fundamental for obtaining a regularity result later on, as it will ensure that \Cref{transform-pde} is uniformly elliptic. We start by establishing the boundedness of $v$. 
\begin{Lemma}\label{lem:bounded}
For every $(t,x)\in\overline\Omega_T$, the inequality $0\le v(t,x)\leq e_\infty(x)\leq 1/8$ holds.
\end{Lemma}
\begin{proof}
For the lower bound, we simply notice that $v(t,x)\ge J(t,x,1/{\rm e})= 0.$ For the upper bound, recall $e_\infty^\prime(x)=1/2-x$ and $e_\infty^{\prime\prime}(x)=-1$, fix an arbitrary admissible control $\alpha\in\Ac$, let $\theta^\alpha\coloneqq \min\{T,\tau^{\alpha,t,x}\}$, and apply It\^o's formula to get
\begin{align*}
e_\infty\big(X^{\alpha,t,x}_{\theta^\smallfont{\alpha}}\big)-e_\infty(x)&=\int_t^{\theta^\smallfont{\alpha}}\sqrt{\alpha_s}e_\infty^\prime(X^{\alpha,t,x}_s) \, \mathrm{d}W_s-\frac12\int_t^{\theta^\smallfont{\alpha}}\alpha_s \, \mathrm{d}s\\
&\leq \int_t^{\theta^\smallfont{\alpha}}\sqrt{\alpha_s}e_\infty^\prime(X^{\alpha,t,x}_s) \, \mathrm{d}W_s-\frac12\int_t^{\theta^\smallfont{\alpha}}\big(1+\log(\alpha_s))\big) \, \mathrm{d}s,
\end{align*}
where we note that $1+\log (y)\le y$ for all $y\ge 0$. 
Since $e_\infty^\prime$ is bounded on $\Omega_T$, and $\alpha$ is non-negative, we can take expectations above and deduce that
\[
e_\infty(x)\geq \E^\P\bigg[e_\infty\big(X^{\alpha,t,x}_{\theta^\smallfont{\alpha}}\big)+\frac12\int_t^{\theta^\smallfont{\alpha}}\big(1+\log(\alpha_s))\big) \, \mathrm{d}s\bigg]\geq \E^\P\bigg[\frac12\int_t^{\theta^\smallfont{\alpha}}\big(1+\log(\alpha_s))\big) \, \mathrm{d}s\bigg]=J(t,x,\alpha).
\]
By the arbitrariness of $\alpha\in\Ac$, this proves the desired result.
\qed
\end{proof}

\medskip
Next, for each $c\geq 0$, define the subset $\mathcal A_c\coloneqq \{\alpha\in\Ac:  \inf_{t\ge 0} \alpha_t\ge c\}$, and notice that $\Ac=\Ac_0$. Then the following proposition shows that the optimal control can be achieved in $\Ac_c$ for $c>0$ small enough.  
\begin{Proposition}\label{prop:elliptic}
For every $(t,x)\in \overline\Omega_T$, it holds that 
\[
v(t,x)=\sup_{\alpha\in\Ac_{\smallfont{1}\smallfont{/}\smallfont{\mathrm{e}}}} J(t,x,\alpha).
\]  
\end{Proposition}
\begin{proof}
Without loss of generality, we only deal with the case $t=0$.  By definition, it suffices to prove $J(0,x,\alpha)\le J(0,x,\beta)$ for every $\alpha\in\Ac$, where $\beta\in\Ac_{1/\mathrm{e}}$ is defined by $\beta_t\coloneqq \max\{\alpha_t,1/\mathrm{e}\}$, $t\geq 0$. Mimicking the proof of the Dambis--Dubins--Schwarz theorem, one may find some probability space on which there exist a Brownian motion $B$ and two stochastic processes $f$, $g$ such that  $f_t\ge 0$,  $g_t=\max\{f_t,1/\mathrm{e}\}$ for all $t\ge 0$ and 
\[
\mbox{Law}(X^{\alpha,0,x})=\mbox{Law}\bigl((x+B_{\tau_\smallfont{t}})_{t\ge 0}\bigr),\; \mbox{Law}(X^{\beta,0,x})=\mbox{Law}\bigl((x+B_{\sigma_\smallfont{t}})_{t\ge 0}\bigr),
\]
where $\tau_t\coloneqq \int_0^t f_s^2\, \mathrm{d} s, \;\mbox{and}\; \sigma_t\coloneqq \int_0^t g_s^2\, \mathrm{d} s,\; \forall t\ge 0.$ Therefore
\[
J(0,x,\alpha)=\E^\P\bigg[\frac12\int_0^{\min\{T,S_\smallfont{\alpha}\}}\big(1+\log(f_s)\big)\, \mathrm{d} s\bigg],\; \mbox{and}\; J(0,x,\beta)=\E^\P\bigg[\frac12\int_0^{\min\{T,S_\smallfont{\beta}\}}\big(1+\log(g_s)\big)\, \mathrm{d} s\bigg], 
\]
where $S_\alpha$ (resp. $S_\beta$) denotes  the first time that the process $(x+B_{\tau_\smallfont{t}})_{t\ge 0}$ (resp. $(x+B_{\sigma_\smallfont{t}})_{t\ge 0}$) exits from $(0,1)$. In particular, it follows that $S_\alpha = \inf\{t\ge 0:  \tau_t\ge S\},\; \mbox{and}\; S_\beta = \inf\{t\ge 0: \sigma_t\ge S\},$ where $S$ is the first exit time of $(x+B_t)_{t\ge 0}$ from $(0,1)$. In what follows, we prove the stronger pathwise result 
\begin{equation}\label{eq:strongresult}
\int_0^{\min\{T,S_\smallfont{\alpha}\}}\big(1+\log(f_s)\big)\, \mathrm{d} s \le \int_0^{\min\{T,S_\smallfont{\beta}\}}\big(1+\log(g_s)\big)\, \mathrm{d} s, 
\end{equation}
which yields immediately
\[
\E^\P\bigg[\int_0^{\min\{T,S_\smallfont{\alpha}\}}\big(1+\log(f_s)\big)\, \mathrm{d} s\bigg]\le \E^\P\bigg[\int_0^{\min\{T,S_\smallfont{\beta}\}}\big(1+\log(g_s)\big)\, \mathrm{d} s\bigg], 
\] 
as desired. In order to get the aforementioned result in \Cref{eq:strongresult}, set $\tilde S_\alpha\coloneqq \min\{T,S_\alpha\}$,  $\tilde S_\beta\coloneqq \min\{T,S_\beta\}$, $F_t\coloneqq f_t^2$ and $G_t\coloneqq g_t^2$, $t\geq 0$. Then one has by assumption
\[
\int_0^{S_{\smallfont{\alpha}}} F_t \, \mathrm{d} t = \int_0^{S_{\smallfont{\beta}}} G_t \, \mathrm{d} t,\; \mbox{and}\; \int_0^{\tilde S_{\smallfont\alpha}} F_t \, \mathrm{d} t \le \int_0^{\tilde S_{\smallfont\beta}} G_t \, \mathrm{d} t. 
\] 
Indeed, the first equality simply comes from the fact that $S_\alpha$ and $S_\beta$ are the right-inverses of $\tau_\cdot$ and $\sigma_\cdot$ at $S$, while the second inequality can be obtained by considering all the possible case, and recalling that we always have $S_\beta\geq S_\alpha$. For convenience, we now define $H\coloneqq F-\mathrm{e}^{-2}$. Notice then that $G_t=\max\big\{F_t,\mathrm{e}^{-2}\big\}=\mathrm{e}^{-2}+H_t^+,\; t\geq 0.$ It follows that
\[
\int_0^{\tilde S_\smallfont{\alpha}} \big(\mathrm{e}^{-2}+H_t^+-H_t^-\big)\, \mathrm{d} t\le\int_0^{\tilde S_\smallfont{\beta}}\big(\mathrm{e}^{-2}+H_t^+\big)\, \mathrm{d} t,
\] 
or equivalently
\begin{equation}  \label{ineq:time-change}
\mathrm{e}^{-2}(\tilde S_\alpha-\tilde S_\beta)+\int_{\tilde S_\smallfont{\beta}}^{\tilde S_\smallfont{\alpha}} H_t^+\, \mathrm{d} t-\int_0^{\tilde S_\smallfont{\alpha}} H_t^-\, \mathrm{d} t\le 0. 
\end{equation}
We now claim that the following inequality holds
\begin{equation} \label{ineq:desire}
\int_0^{\tilde S_\smallfont{\alpha}}  \bigg(1+ \frac{1}{2}\log\big(\mathrm{e}^{-2}+{H_t^+-H_t^-}\big)\bigg)\, \mathrm{d} t\le\int_0^{\tilde S_\smallfont{\beta}}  \bigg(1+ \frac{1}{2}\log\big(\mathrm{e}^{-2}+{H_t^+}\big)\bigg)\, \mathrm{d} t.
\end{equation} 
Indeed, since $H^+$ and $H^-$ have disjoint supports and $1+\log(\mathrm{e}^{-2})/2=0$, we have 
\[
1+\frac 12\log\big(\mathrm{e}^{-2}+{H_t^+-H_t^-}\big)=\bigg(1+\frac 12\log\big(\mathrm{e}^{-2}+{H_t^+}\big)\bigg)+\bigg(1+\frac 12\log\big(\mathrm{e}^{-2}-{H_t^-}\big)\bigg),
\]
and thus proving \Cref{ineq:desire} is equivalent to showing that
\[
\int_{\tilde S_\smallfont{\beta}}^{\tilde S_\smallfont{\alpha}}\bigg(1+\frac 12\log\big(\mathrm{e}^{-2}+{H_t^+}\big)\bigg)\, \mathrm{d} t+\int_0^{\tilde S_{\smallfont\alpha}} \bigg(1+\frac 12\log\big(\mathrm{e}^{-2}-{H_t^-}\big)\bigg)\, \mathrm{d} t\le 0.
\] 
Note further that $x\longmapsto 1+\log(\mathrm{e}^{-2}+x)/2$ is concave, vanishes at $0$, and its derivative at $0$ is $\mathrm{e}^2/2>0$. This implies that \eqref{ineq:desire} follows from $\int_{\tilde S_\smallfont{\beta}}^{\tilde T_\smallfont{\alpha}} H_t^+\, \mathrm{d} t-\int_0^{\tilde S_{\smallfont\alpha}} H_t^-\, \mathrm{d} t\le 0,$ which is clearly ensured by \Cref{ineq:time-change}. Hence, we obtain
\begin{align*}
\int_0^{\tilde S_\smallfont{\alpha}}\big(1+\log(f_t)\big)\, \mathrm{d} t &= \int_0^{\tilde S_\smallfont{\alpha}}  \bigg(1+ \frac{1}{2}\log\big(\mathrm{e}^{-2}+{H_t^+-H_t^-}\big)\bigg)\, \mathrm{d} t\\
&\le\int_0^{\tilde S_\smallfont{\beta}}  \bigg(1+ \frac{1}{2}\log\big(\mathrm{e}^{-2}+{H_t^+}\big)\bigg)\, \mathrm{d} t= \int_0^{\tilde S_\smallfont{\beta}}\big(1+\log(g_t)\big)\, \mathrm{d} t,
\end{align*}
which ends the proof. \qed
\end{proof}

\section{Proof of the main results}\label{sec:viscosity}

\subsection{Regularity of the entropy function}
Summarising the results from \Cref{sec:control}, and using in particular \Cref{prop:elliptic}, we deduce that
\[
v(t,x)=\sup_{\alpha\in\Ac_{\smallfont{1}\smallfont{/}\smallfont{\mathrm{e}}}} J(t,x,\alpha)\eqqcolon w(t,x),\; \forall (t,x)\in \overline\Omega_T,  
\]
where the right-hand-side of the above equality corresponds to an alternative Hamilton--Jacobi--Bellman PDE, which is now uniformly elliptic
\begin{equation} \label{elliptic-pde}
-\partial_tw(t,x) - \frac12\sup_{a\ge 1/\mathrm{e}} \big\{a \partial_{xx}w(t,x) + \log a + 1\big\}=0, \; (t,x) \in \Omega_T,\; w(T,\cdot) = w(\cdot,0) = w(\cdot,1)=0.
\end{equation}
In order to prove the desired regularity for $v$, we will need a comparison result.

\begin{Lemma}\label{lem:comp}
Fix some $c\geq 0$. Let $u$ and $v$ be respectively a concave bounded upper-semicontinuous viscosity sub-solution and a concave bounded lower-semicontinuous viscosity super-solution of 
\begin{equation}\label{eq:dppeqn}
-\partial_tw(t,x) - \frac12\sup_{a\ge c} \big\{a\partial_{xx}w(t,x) + \log(a) +1\big\}=0, \;  (t,x) \in \Omega_T.
\end{equation}
such that $u(t,0)\leq v(t,0)$, $u(t,1)\leq v(t,1)$ and $u(T,x)\leq v(T, x)$ for all $(t,x)\in\overline \Omega_T$. Then $u\le v$ on $\overline \Omega_T$.
\end{Lemma}

\begin{proof}
This is a very standard result for which we can refer to either \citeauthor*{fleming2006controlled} \cite[Theorem V.8.1 and Remark V.8.1]{fleming2006controlled} or \citeauthor*{crandall1992user} \cite[Theorem 8.2]{crandall1992user}. The main point is to notice that for a given $\lambda>0$, given $(X,Y)\in\R^2$ such that
\[
-3\lambda\begin{pmatrix}
1&0\\
0&1
\end{pmatrix}\leq \begin{pmatrix}
X&0\\
0&-Y
\end{pmatrix}\leq 3\lambda \begin{pmatrix}
1&-1\\
-1&1
\end{pmatrix},
\]
if we define 
\[
F_c(q)\coloneqq \sup_{a\ge c} \big\{aq + \log(a) +1\big\}=+\infty\mathbf{1}_{\{q\geq  0\}}+\log(-q)\mathbf{1}_{\{0>q\geq -1/c\}}+\big(cq+\log(c)+1\big)\mathbf{1}_{\{q<-1/c\}},
\]
then  we have $F_c(X)-F_c(Y)\leq 0$ since $F_c$ is non-increasing. It is then direct using the aforementioned references to deduce the desired result.
\qed
\end{proof}

\medskip
Given the previous comparison theorem, the following result is now standard.
\begin{Proposition}
\label{prop:c21}
The function $w$ is the unique bounded continuous viscosity solution to both {\rm \Cref{transform-pde,elliptic-pde}}. Moreover, $w\in C^{1,2}(\Omega_T)$ is concave in $x$.
\end{Proposition}
\begin{proof}
It is standard that $w$ is a (discontinuous) viscosity solution to   \Cref{elliptic-pde}, since we know by \Cref{lem:bounded} that it is bounded, and thus locally bounded. By \Cref{prop:elliptic}, it is also a (discontinuous) viscosity solution to \eqref{transform-pde}. This tells us that the lower-semicontinuous envelope $w_\star$ of $w$ is a viscosity super-solution of both \Cref{transform-pde,elliptic-pde} and that its upper-semicontinuous envelope $w^\star$ is a viscosity sub-solution of \Cref{transform-pde,elliptic-pde}. By \Cref{lem:comp} this proves $w^\star\leq w_\star$, and thus that equality holds, proving that $w$ is a continuous viscosity solution of \Cref{transform-pde,elliptic-pde}. 

\medskip
Concavity is immediate from the viscosity solution property we just proved, as the non-linearity explodes for non-negative values of the second-order derivative. As for the regularity, this comes from the celebrated Evans--Krylov theorem (see for instance \citeauthor*{krylov1984boundedly} \cite[Theorem 1.1]{krylov1984boundedly}), as the operator in \eqref{transform-pde} is concave with respect to the second-order derivative and uniformly elliptic.
\qed
\end{proof}

\subsection{Representation of the entropy function}
In the following, we denote by $e\coloneqq u=w$ the unique classical solution to \eqref{log-pde}--\eqref{tc}--\eqref{bc}, and we investigate its relation to the logarithmic diffusion equation  \eqref{log-pde2}--\eqref{tc2}--\eqref{bc2}. In order to derive the integral representation of $e$ in \Cref{thm:main}, we adopt an approximation argument. Namely, for every $n\in \N^\star$, consider the PDEs
\begin{align}
\label{n-log-pde}
2\partial_te(t,x)&= \log\big(-\partial_{xx}e(t,x)\big), \; (t,x) \in \Omega_T, \\
\label{n-tc}
e(T,x) &= \frac{e_\infty(x)}{n} ,\; x \in (0,1), \\
\label{n-bc}
e(t,0) &= e(t,1) = 0, \;  t \in (0,T),\\
\label{n-log-pde2}
2\partial_tp(t,x)&= \partial_{xx}\big(\log(p(t,x))\big), \; (t,x) \in\Omega, \\
\label{n-tc2}
p(0,x) &= \frac 1 n,\; x \in (0,1), \\
\label{n-bc2}
p(t,0) &= p(t,1) = 1, \;  t \in (0,T).
\end{align}
Then the representation result is summarised in the following proposition. 
\begin{Proposition}\label{prop:n-repre}
There exist a unique classical solution to  \eqref{n-log-pde}--\eqref{n-tc}--\eqref{n-bc} and a unique classical solution to \eqref{n-log-pde2}--\eqref{n-tc2}--\eqref{n-bc2}, denoted respectively by $e^n$ and $p^n$. It furthermore holds that
\begin{equation}
\label{n-candidate}
  e^n(t,x)=-\int_0^x\int_0^y p^n(T-t,z)  \, \mathrm{d} z\, \mathrm{d} y + x\int_0^1\int_0^y p^n(T-t,z)  \, \mathrm{d} z\, \mathrm{d} y,\; (t,x) \in \overline \Omega_T.
\end{equation}
\end{Proposition}

\begin{proof}
As the initial and boundary conditions are strictly positive, by the maximum principle, see for instance \citeauthor*{davis1996some} \cite[Proposition 2.9]{davis1996some}, the PDE \eqref{n-log-pde2}--\eqref{n-tc2}--\eqref{n-bc2} is well-defined in the classical sense and its unique classical solution $p^n>0$ on $\Omega_T$. Define then $e^n: \overline\Omega_T\longrightarrow\R$ by
\[
  e^n(t,x)\coloneqq -\int_0^x\int_0^y p^n(T-t,z)  \, \mathrm{d} z\, \mathrm{d} y + x\int_0^1\int_0^y p^n(T-t,z)  \, \mathrm{d} z\, \mathrm{d} y.
  \]
  It remains to verify that $e^n$ is the unique classical solution to  \eqref{n-log-pde2}--\eqref{n-tc2}--\eqref{n-bc2}.
  A straightforward computation yields $e^n(T,\cdot)=e_\infty/n$, $e^n(\cdot,0)=e^n(\cdot,1)=0$ and  $\partial_{xx}e^n(t,x)=-p^n(T-t,x)$. Furthermore, one has by Fubini's theorem
\begin{align*}
  \partial_te^n(t,x)&=\int_0^x\int_0^y p^n_t(T-t,z)  \, \mathrm{d} z\, \mathrm{d} y - x\int_0^1\int_0^y p^n_t(T-t,z)   \, \mathrm{d} z\, \mathrm{d} y\\
  &=\frac 1 2 \int_0^x\int_0^y  \partial_{xx}\big(\log(p^n(T-t,z))\big)  \, \mathrm{d} z\, \mathrm{d} y -\frac x 2  \int_0^1\int_0^y  \partial_{xx}\big(\log(p^n(T-t,x))\big)  \, \mathrm{d} z\, \mathrm{d} y\\
   &=\frac 1 2 \int_0^x  \partial_x\big(\log(p^n(T-t,y))\big)  \, \mathrm{d} y -\frac x 2   \int_0^1 \partial_x\big(\log(p^n(T-t,y))\big)  \, \mathrm{d} y\\
   &=\frac 1 2  \log\big(p^n(T-t,x)\big) = \frac 1 2  \log\big(-\partial_{xx}e^n(t,x)\big), \; (t,x) \in \Omega_T,
\end{align*}
which implies that $e^n$ is a classical solution to \eqref{n-log-pde}--\eqref{n-tc}--\eqref{n-bc}. \qed
\end{proof}

\medskip
Next, we let $n\longrightarrow\infty$ in \Cref{n-candidate}.
\begin{Proposition}\label{prop:repre}
With the notation of {\rm\Cref{prop:n-repre}}, $n\longmapsto e^n(t,x)$ and $n\longmapsto p^n(t,x)$ are non-increasing and convergent for every $(t,x)\in \overline \Omega_T$. In particular, $e^n$ converges pointwise to $e$ and 
\[
e(t,x)=-\int_0^x\int_0^y \tilde p(T-t,z)  \, \mathrm{d} z\, \mathrm{d} y + x\int_0^1\int_0^y \tilde p(T-t,z)  \, \mathrm{d} z\, \mathrm{d} y,\; (t,x) \in \overline \Omega_T,
\]
where $e$ is the classical solution to \eqref{log-pde}--\eqref{tc}--\eqref{bc} and $\tilde p$ is the pointwise limit of $(p^n)_{n\in\N^\smallfont{\star}}$.  
\end{Proposition}

\begin{proof}
Note that $e^n$ is the unique bounded viscosity solution to \eqref{n-log-pde}--\eqref{n-tc}--\eqref{n-bc}, thus the comparison principle---from a straightforward generalisation of \Cref{lem:comp}---yields the required monotonicity for $(e^n(t,x))_{n\ge 1}$ and thus its pointwise convergence. Furthermore, the stability of viscosity solutions, see \cite[Remark 6.3]{crandall1992user} tells us that its pointwise limit must be $e$ and the convergence is even uniform on $\overline \Omega_T$. 

\medskip

Next, consider $(p^n)_{n\ge 1}$. It follows from \citeauthor*{davis1996some} \cite[Remark 3.1]{davis1996some}, that 
 there exists some $\eps_n>0$ such that $\eps_n<p^n(t,x)\le \max\{1, \|e_\infty\|_\infty/n\}=1$ for all $(t,x) \in \overline\Omega_T$ and $n$ large enough. Fix a sufficiently large $n$ and write for notational simplicity $p^n\equiv u$ and $p^{n+1}\equiv v$. We define the function $a:\overline\Omega_T \longrightarrow \R$ by
\[
a(t,x)\coloneqq \int_0^1 \frac{1}{u(t,x)\xi  + v(t,x)(1-\xi)}  \, \mathrm{d} \xi,
\]
where it follows that $1 \le a(t,x)\le 1/\eps_{n+1}$. Set $w\coloneqq u-v$. Then it holds that
\[
2\partial_t w= \partial_{xx}(aw) = aw_{xx}+2a_xw_{x} +a_{xx}w \mbox{ on } \Omega_T,
\]
and $w(0,\cdot)> 0$, $w(\cdot,0)=w(\cdot,1)=0$. We deduce from the linear maximum principle (see \citeauthor*{imbert2013introduction}, \cite[Theorem 2.2.4]{imbert2013introduction}) that $w\ge 0$. Therefore, the pointwise limit of $(p^n)_{n\in\N^\smallfont{\star}}$ exists and can be denoted by $\tilde p$. We may thus conclude the proof by the dominated convergence theorem. \qed
\end{proof}

\medskip
We are now in a position to prove the main result.

\medskip
\begin{proof}[Proof Theorem \ref{thm:main}]
The existence and uniqueness of a classical solution come from \Cref{prop:c21}, which also establishes the concavity in $x$. Combining the uniqueness with a straightforward verification for $e^\prime(t,x)\eqqcolon e(t,1-x)$, we deduce $e=e^\prime$ and thus the symmetry with respect to $x=1/2$. The bounds are shown in \Cref{lem:bounded} from the control representation\footnote{Alternatively, it can be verified that 0 and $e_\smallfont{\infty}$ are sub- and super-solutions, respectively.}.

\medskip We next prove \Cref{candidate}. As $p^n$ is a classical solution to \eqref{n-log-pde2}--\eqref{n-tc2}--\eqref{n-bc2}, it is also the unique weak solution to \eqref{n-log-pde2}--\eqref{n-tc2}--\eqref{n-bc2}, namely, 
 \[ 
\int_0^T\int_0^1 \big(p^n(t,x)\partial_t\phi -\frac 1 2 \partial_x\big(\log(p^n(t,x))\big)\partial_x\phi(t,x) \big)  \, \mathrm{d} x\, \mathrm{d} t + \int_0^1 \frac{\phi(x,0)}{n}  \, \mathrm{d} x  =0
\]
holds for  all $\phi\in H^1(\Omega_T)\cap C(\overline \Omega_T)$ vanishing at $t=T$, $x=0$ and $x=1$; see \cite[Theorem 2.1]{salvarani2009asymptotic} for the uniqueness of the weak solution.
 Adopting the arguments of \citeauthor*{GS2007} \cite[Section 4]{GS2007}, $(p^n)_{n\in\N^\smallfont{\star}}$ converges weakly in $\L^2(\Omega_T)$ to the unique weak solution $p$ of \eqref{log-pde2}--\eqref{tc2}--\eqref{bc2}. As $\L^2(\Omega_T)$ is a reflexive Banach space, by means of Mazur's lemma, there exists a function $N:\N\longrightarrow\N$ and a sequence of finite sets $\{\alpha (n)_{k} :k\in\{n,\ldots ,N(n)\}\}\subset \R_+$ satisfying 
$\sum _{k=n}^{N(n)}\alpha (n)_{k}=1$ such that 
\[
\lim_{n\to\infty} q^n = p, \mbox{ in } \L^2(\Omega_T),\; \mbox{with } q^n\coloneqq  \sum _{k=n}^{N(n)}\alpha (n)_{k}p^k.
\]
By construction, $(q^n)_{n\in\N^\smallfont{\star}}$ also converges pointwise to $\tilde p$. Hence, $p=\tilde p$ and the desired result \eqref{candidate} follows. Finally, we can deduce the monotonicity of $e$ in time. We established in the proof of \Cref{prop:repre} that $p^n\le 1$, and hence we have $p\le 1$.
From the representation formula, $2\partial_te = \log(-\partial_{xx}e) = \log(p(T-\cdot, \cdot)) \le 0$, as desired.
 \qed
\end{proof}
\medskip

\subsection{Comparison results}

We establish here two comparison results. We will use the second estimate further in Section \ref{sec:asympt} to give asymptotic bounds close to the time boundary.

\begin{Proposition}
\label{lem:qv}
For any $(t,x)\in[0,T]\times[0,1]$, we have
\[
p(t,x)\leq \bar p(t,x), \; \text{\rm or equivalently,} \;
\sigma^\star(t,x)\geq \bar\sigma(t,x),
\]
where $\bar\sigma(t,x)\coloneqq \sin(\pi x)/(\pi\sqrt{T-t})$ is the solution from {\rm\cite{backhoff2023most}}
and $\bar p \coloneqq 1/\bar\sigma^2 = - \bar e_{xx}$, with $\bar e$ given in \eqref{eqn:bare}, and
$p$ the unique weak solution to \eqref{log-pde2}--\eqref{tc2}--\eqref{bc2}.
\end{Proposition}
\begin{proof}
Recall the smooth maps $(p^n)_{n\in\N^\smallfont{\star}}$ from \eqref{n-log-pde2}--\eqref{n-tc2}--\eqref{n-bc2}. Define now $q^n(t,x)\coloneqq -p^n(T-t,x)$. To prove the above claim, it is enough to show that
\[
w^n(t,x)\coloneqq \log(-q^n(t,x))\leq -2\log\big(\hat\sigma(t,x)\big)\eqqcolon w(t,x),\; (t,x)\in[0,T]\times[0,1], 
\]
given the convergence of $p^n$ to $p$. Notice then that by immediate computations, $w^n$ is a classical solution to 
\begin{equation}\label{PDE:quadvar}
-\partial_tw^n(t,x)-\mathrm{e}^{-w^\smallfont{n}(t,x)}\partial_{xx}w(t,x)=0,\; (t,x)\in[0,T)\times(0,1),\; w^n(T,\cdot)=-\log(n),\; w^n(\cdot,0)=w^n(\cdot,1)=0.
\end{equation}
Since $w(\cdot,0)=w(\cdot,1)=+\infty$, $w$ is a classical super-solution of the above PDE. Notice also that $w$ is strictly convex in $x$. We now claim that the desired inequality follows from a comparison theorem between smooth sub-solutions and smooth strictly convex (in $x$) super-solutions to \eqref{PDE:quadvar}. Indeed, when all functions are smooth, we can argue by contradiction as follows. Suppose that the inequality is not true, and follow standard arguments ensuring that we can then find an interior point $(t_o,x_o)\in (0,T)\times (0,1)$ such that there is some $\delta>0$ with
\[
w^n(t_o,x_o)-w(t_o,x_o)\geq \delta,\; \partial_tw^n(t_o,x_o)=\partial_tw(t_o,x_o),\; \partial_{xx}w^n(t_o,x_o)\leq\partial_{xx}w(t_o,x_o).
\]
Using the fact that $w^n$ is a solution (and thus a sub-solution) to \eqref{PDE:quadvar} and $w$ is a super-solution to \eqref{PDE:quadvar} we then deduce that
\[
\mathrm{e}^{w^\smallfont{n}(t_\smallfont{o},x_\smallfont{o})-w(t_\smallfont{o},x_\smallfont{o})} \partial_{xx}w(t_o,x_o)\leq  \partial_{xx}w^n(t_o,x_o).
\]
This is impossible since by positivity of $\partial_{xx}w(t_o,x_o)$, we would then have 
\[
 \partial_{xx}w(t_o,x_o)< \mathrm{e}^{w^\smallfont{n}(t_\smallfont{o},x_\smallfont{o})-w(t_\smallfont{o},x_\smallfont{o})} \partial_{xx}w(t_o,x_o)\leq  \partial_{xx}w^n(t_o,x_o).
 \]
 \qed
\end{proof}

\begin{Proposition}\label{prop:compare}
Let $p_1$ be a weak solution to \eqref{log-pde2}--\eqref{tc2}--\eqref{bc2}
and $p_2 \in C^{1,2}(\Omega_T) \cap C(\overline{\Omega}_T)$, non-decreasing in $t$,
that satisfies \eqref{log-pde2}--\eqref{tc2}, $0 < p_2(t,0) \le 1$ and $0 < p_2(t,1) \le 1$ for $t \in (0,T]$.
Then we have $p_2\le p_1$ in $\overline{\Omega}_T$.
\end{Proposition}
\begin{proof}
We first show the corresponding result for $p_1$ that satisfies \eqref{n-tc2} instead of \eqref{tc2}, so that $p_1 \in C^{1,2}(\Omega_T)$ also.
We suppress the dependence on $n$ for brevity.
We then have
\[
2 \partial_t (p_2-p_1) = \partial_{xx} \log (p_2/p_1), \; (p_2/p_1)(0,x) = 0, \;  (p_2/p_1)(t,0) \le 1, \;  (p_2/p_1)(t,1) \le 1.
\]
Consider now $q_i \coloneqq p_i {\rm e}^{\varepsilon t}$ for some $\varepsilon>0$, $i\in\{1,2\}$, so that
\[
2 {\rm e}^{- \varepsilon t} \partial_t (q_2-q_1) - \varepsilon {\rm e}^{-\varepsilon t} (q_2-q_1)  = \partial_{xx} \log (q_2/q_1), \; (q_2/q_1)(0,x) = 0, \;  (q_2/q_1)(t,0) \le 1, \;  (q_2/q_1)(t,1) \le 1.
\]
We now show that no interior local maximum of $q_2/q_1$ can have $q_2/q_1 > 1$. At such a point, 
we would have $q_2 - q_1 > 0$, $ \partial_{xx} \log (q_2/q_1)\le 0$, and, using the assumption $\partial_t p_2 \ge 0$, from which $\partial_t q_2 \ge 0$
\[
\partial_t (q_2-q_1) = q_1 \left(\frac{\partial_t q_2}{q_1}- \frac{\partial_t q_1}{q_1}\right) \ge q_1 \left(\frac{\partial_t q_2}{q_2}- \frac{\partial_t q_1}{q_1}\right)
= q_2 \partial_t \left(\frac{q_2}{q_1} \right) \ge 0,
\]
leading to a contradiction.
By letting $\varepsilon \longrightarrow 0$, we deduce that $p_2 \le p_1$, for fixed arbitrary $n$.
Finally, we let $n\longrightarrow \infty$ and use the pointwise convergence established in the proof of \Cref{thm:main}.\ref{it_int} to conclude.
 \qed
\end{proof}

\subsection{Convergence to the stationary solution}

Thanks to the representation theorem, we may obtain an estimate of the decay rate to the stationary solution. More precisely, we set $e\equiv e^T$ to emphasise the dependency of  \eqref{log-pde}--\eqref{tc}--\eqref{bc} on $T$ and rewrite the integral representation of $e^T$ as follows
\begin{align*}
 e^T(t,x) &= -\int_0^x  p(T-t,z)  \, \mathrm{d} z \int_z^x  \, \mathrm{d} y + x\int_0^1  p(T-t,z)  \, \mathrm{d} z \int_z^1  \, \mathrm{d} y \\
  & =(1-x)\int_0^x z p(T-t,z)  \, \mathrm{d} z + x\int_x^1 (1-z)p(T-t,z)  \, \mathrm{d} z.  
   \end{align*}
Then we have the following result.
\begin{Corollary}
Let $e^T$ be the solution to {\rm\Cref{log-pde}--\eqref{tc}--\eqref{bc}}. Then it holds for every even number $\alpha\in \N$,
\[ 
|e^T(t,x)-e_\infty(x)| \le   x(1-x)\exp\bigg(\frac{-(\alpha-1)(T-t)}{\pi\alpha^2}\bigg),\; \forall (t,x)\in \overline\Omega_T.
\]
\end{Corollary}

\begin{proof}
First, note that the stationary solution to \eqref{log-pde2}--\eqref{tc2} is $p_\infty(x)=1$. Similarly, denote by $p\equiv p^T$ the weak solution to \eqref{log-pde2}--\eqref{tc2}--\eqref{bc2}. By means of \cite[Theorem 4.1]{salvarani2009asymptotic},  there exists some $C>0$ such that for any even number $\alpha\in \N$
\[ 
\|p^T(t,\cdot)-p_\infty\|_{\L^\smallfont{\alpha}([0,1])} \le \exp\bigg(\frac{-C(\alpha-1)t}{\alpha^2}\bigg),\; \forall t\in [0,T].
\]
Inspection of the proof of \cite[Theorem 4.1]{salvarani2009asymptotic} reveals that $C$ is the sharp constant of Poincaré's inequality of the relevant domain, which for $[0,1]$ is 
known explicitly as $1/\pi$.
Hence, choosing $\beta$ such that $\frac 1 \alpha +  \frac 1 \beta=1$, H\"older's inequality allows to conclude
\begin{align*}
|e^T(t,x)-e_\infty(x)| &= \bigg |(1-x)\int_0^x z (p^T-p_\infty)(T-t,z)  \, \mathrm{d} z + x\int_x^1 (1-z) (p^T-p_\infty)(T-t,z)  \, \mathrm{d} z\bigg| \\
&\le (\beta+1)^{-1/\beta}x(1-x)\exp\bigg(\frac{-C(\alpha-1)(T-t)}{\alpha^2}\bigg)\big[x^{1/\beta}+(1-x)^{1/\beta}\big]  \\
&\le x(1-x)\exp\bigg(\frac{-C(\alpha-1)(T-t)}{\alpha^2}\bigg).
\end{align*}
\qed
\end{proof}

\subsection{Well-posedness for the optimal martingale}

In this section we provide the proof of \Cref{th:wellposedsde}. This requires the following technical result.
\begin{Lemma}
\label{lem:prop_sigma}
For all $\varepsilon>0$, 
there exist positive constants $C_{\varepsilon}$ and $D_\varepsilon>0$ such that
\begin{align}
\label{bounds}
1/\mathrm{\sqrt{e}} \le \sigma^\star(t,x) &\le C_\varepsilon,  \; \forall (t,x) \in [0,T-\varepsilon] \times [0,1], \\
|\sigma^\star(t,x) - \sigma^\star(t,y)| &\le D_\varepsilon |x-y|^{1/2}, \;
 \forall (t,x,y) \in [0,T-\epsilon] \times [0,1]^2.
\label{holder}
\end{align}
\end{Lemma}
\begin{proof}
Fix $\eps>0$. From \Cref{def:weak}.$(ii)$ and \Cref{thm:main}.$(ii)$, by the classical continuous embedding of $H^1$ into the space of $1/2$--H{\"o}lder continuous functions
(see, \emph{e.g.}, \cite[Theorem 4.12]{adams2003sobolev}), we have for some
$c_\varepsilon>0$ that
\begin{eqnarray}
\label{eq:holderp}
|p(t,x) - p(t,y)| \le c_\varepsilon |x-y|^{1/2},\; \forall (t,x,y) \in [0,T-\varepsilon] \times [0,1]^2.
\end{eqnarray}
We also have from \Cref{thm:main}.$(ii)$ that $\partial_{xx} e(t,x) = - p(T-t,x)$, where $e$ is a classical solution to \eqref{log-pde}. This can only be the case
if there exists $\tilde c_{\varepsilon}>0$ such that $-\partial_{xx} e(t,x) \ge \tilde c_{\varepsilon}$ for all $(t,x) \in [0,T-\varepsilon] \times [0,1]$.
Otherwise, there would be a sequence $(t_n,x_n)_{n\in\N}$, valued in $[0,T-\varepsilon] \times [0,1]$, with $(t_n,x_n) \longrightarrow (t_o,x_o) \in [0,T-\epsilon] \times [0,1]$ with
$p(t_n,x_n) \longrightarrow 0$ as $n\longrightarrow \infty$. Since $p$ assumes its positive boundary value \eqref{bc2}
continuously by virtue of \Cref{def:weak}.$(ii)$, we cannot have $x_o \in \{0,1\}$.
But if $x_0 \in (0,1)$, then $-\partial_{xx} e(t_n,x_n) = p(t_n, x_n) \longrightarrow 0$ implies by \eqref{log-pde} that  $|\partial_t e(t,x)(t_n,x_n)| \longrightarrow \infty$ as $n\longrightarrow \infty$, which is a contradiction to $e$ being a classical solution with locally bounded derivatives in the interior. 

\medskip
From $p(T-t,x)\ge c_{\varepsilon}$ and \eqref{eq:holderp} we can deduce \eqref{holder}.
Moreover, we get the upper bound in \eqref{bounds} taking $C_\varepsilon = 1/c_\varepsilon$.
Finally, we have from \Cref{prop:elliptic} that $\sigma^\star(t,x) \ge 1/\mathrm{\sqrt{e}}$ for all $x$ and $t$ and hence the lower bound in \eqref{bounds}.
\qed
\end{proof}

\medskip

\begin{proof}[Proof of Theorem \ref{th:wellposedsde}]
The uniqueness of solutions follows from \citeauthor*{yamada1971uniqueness} \cite{yamada1971uniqueness}, see \citeauthor*{karatzas1991brownian} \cite[Proposition 5.2.13 and Example 5.2.14]{karatzas1991brownian}, by the $1/2$--H{\"older} continuity proved in \Cref{lem:prop_sigma}, recalling \Cref{holder}.
The existence is then implied by \citeauthor*{gyongy1996existence} \cite[Theorem 2.4]{gyongy1996existence} (see also the revised version in \citeauthor*{gyongy2022existence} \cite[Theorem 2.1]{gyongy2022existence}, as well as \citeauthor*{gyongy2011note} \cite[Remark 1.1]{gyongy2011note}, or the earlier results of \citeauthor*{veretennikov1981strong} \cite{veretennikov1981strong}.).
Finally, the last statement of the theorem 
follows by the same arguments as for Lemma 5.1 at the start of Section 5 in \cite{backhoff2023most}.
\end{proof}

\section{Approximation scheme}\label{sec:numerics}

We discretise the PDE in the form \eqref{elliptic-pde} using $M\in\N^\star$ time points and a time step $k\coloneqq T/M$, as well as $N\in\N^\star$ spatial intervals of width $h\coloneqq 1/N$.
We write $v_n^m$ for the approximation to $w(m k, n h)$, for $n\in\{0,\ldots,N\}$, $m\in\{0,\ldots,M\}$. The boundary conditions are then for all $n\in\{0,\ldots,N\}$, $v_n^M \coloneqq  0$, and for all $m\in\{0,\ldots,M\}$, $v_0^m \coloneqq  0$, $v_N^m \coloneqq  0$.
We first introduce a regularised problem with strictly positive and bounded control set:
for any positive constant $d\ge 1/{\rm e}$, set $I^d\coloneqq  [1/{\rm e},d]$. 
Considering $w_d$ the solution of \eqref{elliptic-pde} with $\sup_{a \ge 1/{\rm e}}$ replaced by $\sup_{a \in [1/{\rm e},d]}$, the dominated convergence theorem ensures the pointwise convergence of $w^d$ to $w$.


\medskip
We will discuss both explicit and implicit time-stepping schemes. For the explicit finite difference scheme, which is understood backwards in time, with $v^M=0$, and for all $m\in\{1,\ldots,M\}$, $n\in\{1,\ldots,N-1\}$
\begin{equation}
\label{expl_scheme}
2 \frac{v_n^{m}-v_n^{m-1}}{k} = \inf_{a\in I^\smallfont{d}} \big\{(-a (A v^m)_n - \log a - 1\big\},
\end{equation}
where $v^m \coloneqq  (v_0^m,\ldots, v_N^m)$, and the matrix operator $A$ is defined row-wise for $n\in\{1,\ldots,N-1\}$ as $(A v^m)_n \coloneqq  (v_{n+1}^m - 2 v_{n}^m + v_{n-1}^m)/h^2$.

\medskip
The explicit scheme can be re-arranged as
\begin{equation}
\label{expl_scheme_2}
v_n^{m-1} =  \sup_{a \in  I^\smallfont{d}} \big\{ \pi(a)  v^m_{n+1} + (1-2\pi(a))  v^m_{n}   + \pi(a)  v^m_{n-1}   + k (\log a + 1)/2 \big\},
\end{equation}
for $\pi(a) \coloneqq  k a/ (2 h^2)$. If $k d/ h^2 \le 1$, $\pi(a)$ and $1-2\pi(a)$  are guaranteed to be non-negative for all $a \in I^d$ and are interpretable as transition probabilities. 
Therefore, defining a symmetric random walk by $\widehat{X}_{m} = \widehat{X}_{m-1} + h \xi_{m-1},$ where $\xi_m$ are i.i.d. with $\mathbb{P}[\xi_m = 1] = \mathbb{P}[\xi_m = -1] =  \pi(a)$, $\mathbb{P}[\xi_m = 0]= 1-2 \pi(a)$, and $0$ else, we have
\begin{equation*}
v_n^0 = \frac{1}{2} \sup_{\hat{a}} \bigg\{\sum_{j=0}^{\hat{\tau}-1} (\log(a_j) + 1) k\bigg\}, \; \hat{\tau} \coloneqq  \min\{j\in\N: \widehat{X}_{j} \in \{0,1\} \},
\end{equation*}
where $\hat{a} \coloneqq  (a_0,\ldots, a_{M-1})$ is an admissible discrete control process. 
By choosing $a_j=1/\mathrm{e}$ for all $j$, it is clear that $v^m$ is non-negative for all $m$. Moreover, $x (1-x)/2$ is a super-solution to the scheme, from which is follows that
$v_n^m \le x_n (1-x_n)/2$ for all $n$ and $m$.

\medskip
We now turn to the implicit scheme.
For all $m\in\{1,\ldots,M\}$, $n\in\{1,\ldots,N-1\}$, let
\begin{equation}
\label{scheme}
2 \frac{u_n^{m+1}-u_n^m}{k} =  \inf_{a \in I^\smallfont{d}} \big\{-a (A u^m)_n - \log a - 1\big\},
\end{equation}
where $u^m \coloneqq  (u_0^m,\ldots, u_N^m)$. This can be written as
\begin{equation}
\label{scheme_2}
 \inf_{a \in I^\smallfont{d}} \big\{ ((1- (k a)/2 A) u^m)_n - k (\log a +1)/2  \big\} = u_n^{m+1}.
\end{equation}
Using that $1- (k a)/2  A$ is a (strictly diagonally dominant) $M$-matrix, we have that the scheme is monotone and a similar argument to above gives the same bounds on the solution
as for the explicit scheme, without constraints on the time-step. Therefore, we can let $d \uparrow\infty$ and obtain monotone convergence of the discrete solution. 

\medskip
A standard calculation shows that the explicit and implicit scheme are consistent with the PDE. The framework by \citeauthor*{barles1991convergence} \cite{barles1991convergence} then implies convergence to the viscosity solution of the PDE as $k, h \downarrow 0$, maintaining $k d/h^2 \le 1$ in the case of the explicit scheme. We will focus on the implicit scheme from now on for its unconditional stability, which is here especially useful due to the arbitrarily large control values close to the terminal time, \emph{i.e.}, for convergence we need to choose arbitrarily large $d$.

\medskip
The system \eqref{scheme_2}, combined with boundary conditions, is a nonlinear finite dimensional system of equations, which can be solved by policy iteration: 
starting from an initial guess $u^{(0)}$, define for each $i\in\N$, $a^{(i)}_n \coloneqq  \min\big\{- 1/(A u^{(i)})_n , d\big\},$ and then solve the \emph{linear} system
\begin{equation*}
\label{it_scheme}
u_n^{(i+1)} - k/2  \big(a^{(i)}_n (A u^{(i+1)})_n + \log a^{(i)}_n + 1\big) = u_n^{m+1}.
\end{equation*}
This iteration converges super-linearly by standard results (see \citeauthor*{bokanowski2009some} \cite{bokanowski2009some}). In practice, $2$ or $3$ iterations are sufficient for high accuracy. \Cref{fig:uxx-vol} shows the second derivative of the value function, $\partial_{xx}e$, and the optimal volatility $\sigma^\star$ as function of $x$ for different $t$,
with $T=1$. The numerical solution was computed with $M=N=1000$, $d=10^6$.

\begin{figure}[ht!]
\hspace{-0.2 cm}
\includegraphics[width=0.52\columnwidth]{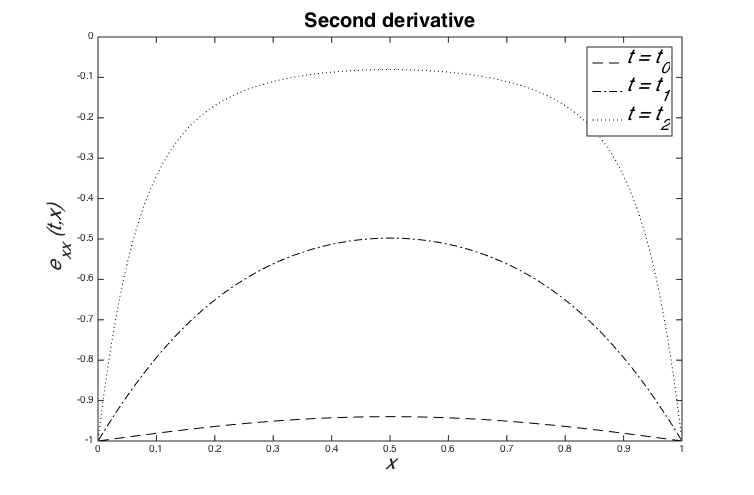}
\hspace{-0.8 cm}
\includegraphics[width=0.52\columnwidth]{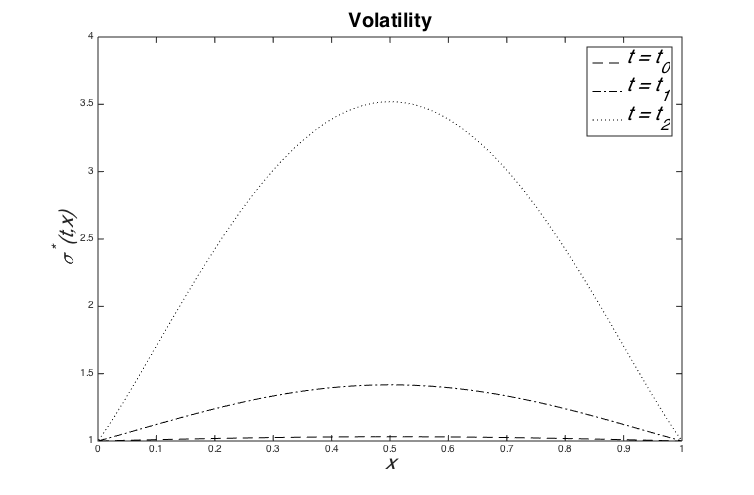}
\caption{\small Second derivative $\partial_{xx}e$ of value function and optimal volatility $\sigma^\star$ for $t_0=0.5$, $t_1=0.9$, $t_2=0.99$.}
\label{fig:uxx-vol}
\end{figure}

\section{Asymptotics near the time boundary}
\label{sec:asympt}

Here, we investigate the behaviour of the solution as $t\longrightarrow T$ by the method of of matched  asymptotic expansions (see, \emph{e.g.}, \cite[Chapter 5]{hinchperturbation}). 
For small values of $T-t$, the solution domain can be separated into 
an `inner region' near each of the boundaries at $x=0,1$,  and an  `outer region' away from the boundaries. We construct separate approximations (which happen to be exact solutions of the PDE but none of which satisfies all the necessary boundary conditions), and join them by matching 
in  overlap regions between the two inner regions and the outer region. We can use these separate solutions to construct a composite expansion which is a uniform approximation in the whole domain.

\medskip
It is more convenient to work with $p$ defined by $p(t,x)\coloneqq -\partial_{xx}e(T-t,x)$, which satisfies
\begin{align}
\label{p_pde}
2 \partial_t p = \partial_x (\partial_x p/p), \; p(0,x) = 0, \; p(t,0) = p(t,1) = 1.
\end{align}
Note that this is now a forward equation with initial condition at $t=0$.


\subsection{`Inner region'}


We begin by considering the inner region near $x=0$. The PDE is invariant under scaling $t$ with $\varepsilon$---real, positive, and thought of as being small---and $x$  with $\sqrt{\varepsilon}$, while the initial condition and the boundary condition at $x = 0$ (but not that at $x = 1$) are invariant under this scaling. We deduce that, for small $t$, there is a region $x/\sqrt{t}=O(1)$ in which  $x$ itself is small but the PDE does not become trivial. For small $t$, therefore, we work in this inner region, thereby ignoring the boundary condition at $x=1$ (it is replaced by asymptotic matching) and solve, in the semi-infinite domain $(t,x)\in(0,+\infty)^2$, the problem  
\begin{align}
\label{p_inner}
2\partial_t p^{\mathrm{in}} = \partial_x (\partial_x p^{\mathrm{in}}/p^{\mathrm{in}}), \; p^{\mathrm{in}}(0,x) = 0, \; 
p^{\mathrm{in}}(t,0) = 1.
\end{align}
The inner solution near $x=1$ follows by substitution $x \longrightarrow 1-x$.
Although we have yet to consider the outer region, we anticipate that $p^{\mathrm{in}}(t,x)\longrightarrow 0$ as $x \longrightarrow \infty$.

\medskip
The whole problem \eqref{p_inner} is invariant under the scaling in $t$ and $x$ mentioned above, which suggests the similarity \emph{ansatz} $p^{\mathrm{in}}(t,x) = f(\xi)$, where $\xi \coloneqq x/\sqrt{t}$. This matches the initial and boundary condition in \eqref{p_inner} if we require
$f(0) = 1$ and $\lim_{\xi\rightarrow \infty} f(\xi) = 0$.
Insertion in the PDE in \eqref{p_inner} gives
\begin{align}
\label{f_inner}
- \xi f^\prime = (f^\prime/f)^\prime.
\end{align}

This equation too has scaling (group) invariances, of which the invariance under scaling $x$ with $\lambda$ and $f$ with $1/\lambda^2$ means that setting $f = g/\xi^2$ leads to the homogeneous problem
\begin{align*}
- \xi g^\prime+ 2 g = \xi^2 (g^\prime/g)^\prime + 2.
\end{align*}
Exploiting the homogeneity by using the logarithmic derivative $\xi\mathrm{d}/\mathrm{d}\xi$ leads to an autonomous second-order equation which can be reduced to a first-order equation---and, in fact, a quadrature---by the further substitution  $\xi g^\prime= g H(g)$, in which the first $g$ on the right is for convenience. The result is the separable equation
\begin{align}
\label{H_ode}
g H H^\prime + (1-g)(2-H) = 0,
\end{align}
with $g(\xi)$ subsequently recovered by separating 
$\xi g^\prime = gH(g)$. 
 A graphical analysis---equivalent to a phase-plane---which we defer to later (see \Cref{sec:asympphase}), shows that the solution we require must satisfy the boundary and asymptotic conditions
 $g(0) = 0$,
 $g(\xi)\longrightarrow 1$ as $\xi\longrightarrow \infty$, and $H(0) = 2$, $H(1) = 0$.

\medskip
 Then by separation and integration of \eqref{H_ode}, and using these conditions, we find
  \begin{align}
\label{H_eqn}
H(g) + 2 \log(2-H(g)) = \log g - g + 1 + 2 \log 2.
\end{align}
 With $H$ implicitly given by \eqref{H_eqn}, using $\xi g^\prime= g H(g)$ and $g = \xi^2 f$
   \begin{align}
\label{f_ode}
f^\prime= - f \cdot \bigg(\frac{2-H(\xi^2 f)}{\xi}\bigg), \; f(0) = 1.
\end{align}
Then $p^{\mathrm{in}}(t,x) = f(x/\sqrt{t})$.

\subsection{`Outer region'}

When considering the outer region, with outer solution 
$p^{\mathrm{out}}(t,x)$, we ignore the spatial boundary conditions as the solution changes rapidly near the boundaries, in the manner just described, and represented by $p^{\mathrm{in}}(t,x)$. We note that the far-field behaviour of $p^{\mathrm{in}}(t,x)$  is $t/x^2$, as $f(\xi) \sim 1/\xi^2$ for large $\xi$. This suggests that we look for separable solutions of the form $p^{\mathrm{out}}(t,x) = a(t) q(x)$. Insertion in \eqref{p_pde} gives 
\begin{align}
\dot{a} = 1 , \; 2 q = (q^\prime/q)^\prime,
\end{align}
since the separation constant may be taken equal to $1$ without loss of generality, by rescaling $q$.
 
\medskip
Bearing in mind the initial condition, we take the constant of integration for $a$ to be zero, so that $a(t)=t$. The equation for $q$ can be solved in various ways (for example, it is homogeneous, so one can first find $q^\prime$ as a function of $q$); its general solution is $q(x) = c^2/\sin^2(c (x-x_0))$ for arbitrary constants $c$ and $x_0$. We now determine these by matching.

\medskip 
As an aside, we note that taking the limit $c\longrightarrow 0$ shows that 
$t/x^2$ is an exact solution to \eqref{p_pde}, without the boundary conditions.

\subsection{Asymptotic matching and comparison}
 
We are now in a position to join the inner and outer solutions (in practice, of course, they are found iteratively in parallel). We need to use the standard Van Dyke matching rule
(see, \emph{e.g.}, \cite[Subsection 5.1.5]{hinchperturbation}) in the form
\[
\text{one-term outer(one-term inner) = one-term inner(one-term outer)}
\]
where `one-term' simply means we take the leading-order behaviour in the limits that correspond to transiting from one region to another; thus on the left we take the behaviour of $f(\xi)$ as $\xi\longrightarrow\infty$, and write it in terms of $x$ and $t$, while on the right we take the behaviour of the outer solution as $x\longrightarrow 0$. The method is equivalent to matching in an overlap region $\sqrt{t} \ll x \ll O(1)$.  

\medskip
From $g(\xi)\longrightarrow 1$ as $\xi\longrightarrow \infty$, $f(\xi) \sim 1/\xi^2$ for large $\xi$. In outer variables this is $t/x^2$, and so matching dictates that $p^{\mathrm{out}}(t,x) \sim t/x^2$ for $x \longrightarrow 0$. 
Expanding $p^{\mathrm{out}}(t,x)= t c^2/\sin^2(c(x- x_0))$ for small $x$, we see that we can only match with $t/x^2$ if $x_0=0$. 

\medskip
At  this point, $c$ is not determined. However, the outer solution must have a similar singularity at $x=1$, and this tells us that $\sin c =0$, so $c=\pi$ (other values like $2\pi$ are ruled out because the solution is strictly concave). It is then automatic---by symmetry---that $p^{\mathrm{out}}(t,x)$ has the correct singularity to match with the inner solution near $x = 1$. 

\medskip
Note that the outer solution coincides with the solution derived in \cite{backhoff2023most}. 
We can now deduce the following comparisons.

\begin{Proposition}\label{prop:bounds}
We have, for all $x\in [0,1]$, $t\in [0,T)$,
\begin{align}
\label{ineq:p}
p^{\mathrm{in}}(t,x) &\le p(t,x) \le p^{\mathrm{out}}(t,x), \\
\label{ineq:e}
e^{\mathrm{out}}(t,x)&  \le  e(t,x) \le \;\; e^{\mathrm{in}}(t,x), \\
\label{ineq:sig}
\frac{\sin(\pi x)}{\pi \sqrt{T-t}} 
&\le \sigma^\star(t,x) \le \frac{\min\{x,1-x\}}{\sqrt{T-t}} \bigg(1 + o\bigg(\frac{\sqrt{T-t}}{\min\{x,1-x\}}\bigg)\bigg).
\end{align}

\end{Proposition}

\begin{proof}
We begin by noting that the right-hand inequality of \eqref{ineq:p} and left-hand inequality of \eqref{ineq:sig} are implied by
Proposition \ref{lem:qv}. We proceed with the left-hand inequality of \eqref{ineq:p}.
By construction, $p^{\mathrm{in}}$ solves \eqref{p_inner}. Moreover, $f$ and hence $p^{\mathrm{in}}$ is smooth and from \eqref{f_ode} with $\xi \ge 0$, $H\le 2$,
we have $f'\le 0$, so that $\partial_t p^{\mathrm{in}}\ge 0$.
As moreover $f\le 1$, also
$p^{\mathrm{in}}(t,1)\le 1$ and 
we can apply \Cref{prop:compare} to deduce $p^{\mathrm{in}}\le p$.

\medskip
Then \eqref{ineq:e} follows by integrating $\partial_te(t,x) = \log(p(T-t,x))$ backwards in $t$, and correspondingly for $e^{\mathrm{in}}$ and $e^{\mathrm{out}}$.

\medskip
Lastly, for the right-hand inequality in \eqref{ineq:sig} we use $\xi^2 f(\xi) \longrightarrow 1$ as $\xi \longrightarrow \infty$, so that
\[
p(t,x) \ge p^{\mathrm{in}}(t,x) \ge \frac{x^2}{t} \frac{1}{1 + o\big(\frac{t}{x^\smallfont{2}}\big)}, \; \text{as} \; \frac{t}{x^2} \longrightarrow 0 
\Longrightarrow  \sigma^\star(t,x) \le \frac{x}{\sqrt{T-t}} \bigg(1 + o\bigg(\frac{\sqrt{T-t}}{x}\bigg)\bigg).
\]
By symmetry in $x$,  \eqref{ineq:sig} follows.
\qed
\end{proof}

\medskip
A plot of $e$, $e^{\mathrm{in}}$ and $e^{\mathrm{out}}$ is given in Fig.\ \ref{fig:asympt_e}, left, illustrating the ordering in \eqref{ineq:e}.
The bounds in \eqref{ineq:sig} show that while $\bar \sigma$ from \cite{backhoff2023most}
is dominated by $\sigma^\star$, they share the essentially same inverse square root singularity close to the terminal time.

\subsection{Composite expansion}

Finally, we construct a composite expansion, in the form `inner + outer $-$ common', where `common' means the part of both expansions that is determined by matching (here, $t/x^2$). 
Adding the inner solutions, outer solution, and subtracting the common limit, we finally get
   \begin{align}
\label{composite}
p(t,x) \sim p^{\mathrm{comp}}(t,x)=
f\left(\frac{x}{\sqrt{t}} \right)
- \frac{t}{x^2} + \frac{\pi^2 t}{\sin^2(\pi x)}  + 
f\left(\frac{1-x}{\sqrt{t}} \right) - \frac{t}{(1-x)^2},
\; t \ll 1, \; 0 \leq x \leq 1,
\end{align}
where $f$ satisfies \eqref{f_ode} with $H$ defined implicitly by \eqref{H_eqn}.

\medskip
For the computations, we solve the ODE \eqref{f_ode} with Matlab's built-in {\tt ode45}, which is based on the explicit Runge--Kutta (4,5) formula, see \citeauthor*{dormand1980family} \cite{dormand1980family}. Herein, the function $H$ is found from \eqref{H_eqn}, where the initial guess provided to the iterative {\tt fzero} solver is chosen as
$2-2 \sqrt{g}$.

\begin{figure}[ht!]
\hspace{-0.6 cm}
\includegraphics[width=0.52\columnwidth, height = 0.4\columnwidth]{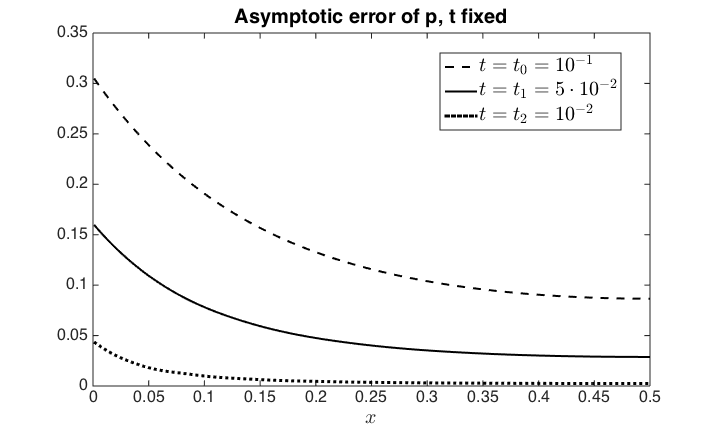}
\hspace{-0.5 cm}
\includegraphics[width=0.52\columnwidth, height = 0.4\columnwidth]{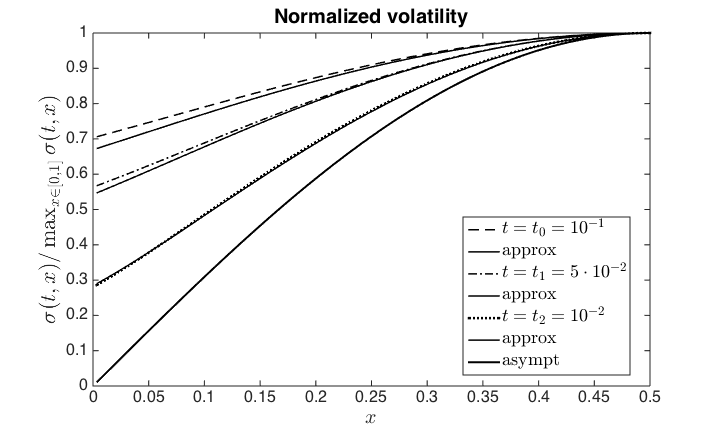}
\caption{\small Left: error between the composite asymptotic solution for $p$ in \eqref{composite} for different $t$, compared to the numerical approximation.
Right: the derived asymptotic solution for $\sigma(t,\cdot) = 1/\sqrt{p(t,\cdot)}$  for different $t$, compared to the numerical approximation, all normalised by their maximum.}
\label{fig:asympt_p}
\end{figure}

The entropy $e$ can now be found by integrating the relationship $2 (\partial_t e)(T-t,x) = \log p(t,x)$ in time.
For the inner region, integrating by parts twice
\begin{align*}
2 e^{\rm in}(T-t,x) - 0 &= - \int_0^t \log p^{\mathrm{in}}(s,x) \, {\rm d}s \\
&= - \int_{x/\sqrt{t}}^\infty \log f(\xi) \frac{2 x^2}{\xi^3}  \, {\rm d}\xi \\
&= x^2  \frac{\log f(\xi)}{\xi^2} \bigg\vert_{x/\sqrt{t}}^\infty +  x^2  \frac{f^\prime(\xi)}{\xi f(\xi)} \bigg\vert_{x/\sqrt{t}}^\infty
- x^2  \int_{x/\sqrt{t}}^\infty \bigg(\frac{f^\prime(\xi)}{f(\xi)}\bigg)^\prime \frac{1}{\xi}  \, {\rm d}\xi  \\
&= -t \log f(x/\sqrt{t}) - x \sqrt{t} \frac{f^\prime(x/\sqrt{t})}{f(x/\sqrt{t})} - x^2 f(x/\sqrt{t}),
\end{align*}
where in the last step we used \eqref{f_inner}. 
%
For the outer region, similarly and more simply,
\begin{align*}
2 e^{\rm out}(T-t,x) = - 2 t \log \pi + t 
- t \log t 
+ 2 t \log(\sin \pi x).
\end{align*}
In the overlap region, $e \sim t - t \log t + 2 t \log(x)$,
so that the composite expansion is
\begin{align}
\label{e_asympt}
e(t,x) \sim e^{\rm comp}(t,x) =  e^{\rm in}(t,x) + 2 t \log(\sin \pi x/ (\pi x))  + 2 t \log(\sin \pi (1-x)/ (\pi (1-x))) + e^{\rm in}(t,1-x).
\end{align}

\begin{figure}[ht!]
\hspace{-0.4 cm}
\includegraphics[width=0.62\columnwidth]{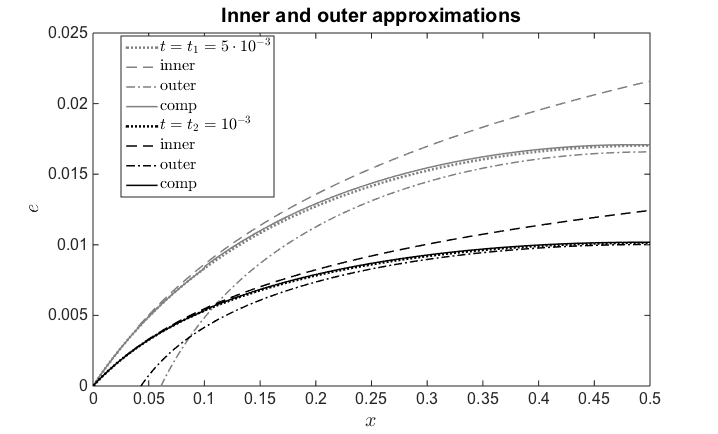}
\hspace{-1 cm}
\includegraphics[width=0.55\columnwidth, trim=2.5cm 0 0 0, clip]{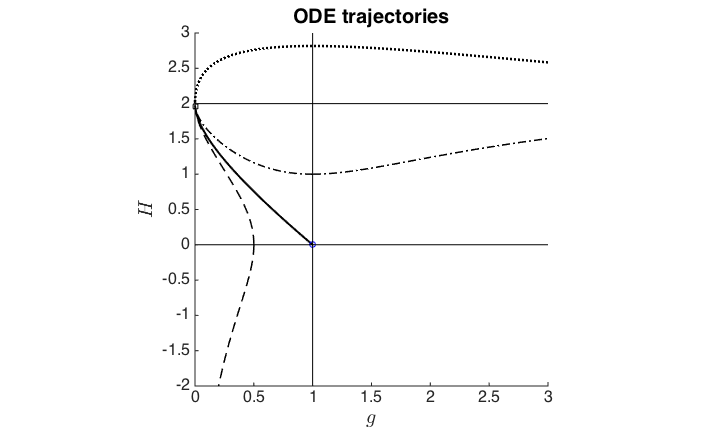}
\caption{\small Left: inner, outer, and composite asymptotic solution for $e$ from \eqref{e_asympt} for different $t$, compared to the numerical approximation.
Right: candidate trajectories of \eqref{H_ode} in the $(g,H)$-plane. The unique path consistent with the boundary conditions is shown as solid.
}
\label{fig:asympt_e}
\end{figure}

In \Cref{fig:asympt_p}, we plot the error in 
$p^{\mathrm{comp}
}(t,x)$, defined as the difference between it and the numerical solution, at a variety of times. The agreement is excellent for small times, and although it appears somewhat large near the boundaries  when $t=10^{-1}$, it is worth noting that taking the limit $x\longrightarrow 0$ in $p^{\mathrm{comp}}(t,x)$ gives the value $1 - \pi^2 t / 3 \approx 1 - 3.29t$, with a very small correction due to the inner and common terms from the boundary $x=1$ (which almost cancel).  This is consistent with the plot. Unfortunately it does not appear easy to construct the second term in the asymptotic expansions.  

\medskip
\Cref{fig:asympt_e} shows the same comparison for the entropy. Being integrated in time, the error is an order of magnitude smaller and the approximation is excellent on the whole domain (here, we expect the error to be $O(t^2)$).\footnote{Note though that the values of $t$ are different between $p$ and $e$. It proved difficult to approximate $p$ numerically to sufficient accuracy for very small $t$ in order to compute the expansion error reliably.}


\subsection{Analysis of the inner-layer ODE}\label{sec:asympphase}

We now look in more detail at the ODE for $H(g)$ that arises in the inner-region analysis. From $f(0) = 1$ it follows that $g(\xi) \sim \xi^2$ as $\xi \longrightarrow 0$. Hence, from \eqref{H_ode}, all relevant trajectories in the $(g,H)$-plane---see \Cref{fig:asympt_e}, right---start from $(0,2)$, which is a node, from which all trajectories emerge tangential to the $H$ axis, with the exception of the trajectory $H=2$  (which corresponds to $g=\xi^2$, $f = 1$, which is a trivial and irrelevant  solution of the original PDE).
We argue that the only possible trajectory is the one that decreases in $H$ and ends in $(1,0)$, which is a saddle with asymptotic directions along $(1, \pm \sqrt{2})$.


\medskip
All trajectories with $g<0$ are irrelevant, as are those that start upwards from $(0,2)$; these latter ones (dotted in \Cref{fig:asympt_e}, right) cannot cross the trajectory $H=2$ and thus $H>2$ on them, from which $\xi g^\prime(\xi)> 2g$, so $(\log g - 2\log \xi)^\prime >0$ and finally $g > c \xi^2$ for some positive $c$, so $f(\xi)>c$ which is inconsistent with the boundary condition at infinity.


 
\medskip
We now consider the trajectories that decrease from $(0,2)$ until they meet the vertical line from $(1,0)$ to $(1,2)$, which they cross  horizontally and thereafter increase, approaching $H=2$ exponentially (dash-dotted in \Cref{fig:asympt_e}, right). Hence there is $c>0$ such that for sufficiently large $g$,
\[
H > 2 - c/g \Longrightarrow  \xi g^\prime/g > 2 - c/g \Longrightarrow \frac{g^\prime}{g-c/2} >   \frac{2}{\xi}.
\]
Integrating, we have $g(\xi) > c/2 + A \xi^2$ for some $A>0$, leading again to a contradiction.

\medskip
Lastly, the trajectories that cross the horizontal line from $(0,0)$ to $(1,0)$ do so vertically, and they are asymptotic to the negative $H$ axis as $g\longrightarrow 0$
(dashed in \Cref{fig:asympt_e}, right). Specifically, as $H \longrightarrow -\infty$ and $g \downarrow 0$, we have $gH^\prime \sim 1$, so that $H \sim \log (g/c_1)$ for some positive $c_1$. Using this in $\xi g^\prime(\xi) = gH(g)$ and integrating gives $\log(c_2\xi) \sim \log\log(g/c_1)$ for some  
$c_2$ which is also positive as $\xi \longrightarrow +\infty$ in this limit. Exponentiating twice gives $g(\xi) \sim c_1 \mathrm{e}^{c_2\xi}$
 as $\xi\longrightarrow\infty$. With its exponential growth, this solution cannot be matched to any outer solution that makes sense in the context of our problem; hence these solutions are inadmissible. We conclude by elimination that our trajectory is, as claimed, that heteroclinic one that joins the two critical points.  
 

\bibliography{bibliography}

\end{document}